\documentclass[a4paper]{amsart}

\usepackage{pstricks}
\usepackage{subfigure}
\usepackage{color}
\usepackage{amsmath,amssymb}
\usepackage[all]{xy}
\usepackage{url}

\newtheorem{theorem}{Theorem}[section]
\newtheorem{lemma}[theorem]{Lemma}

\newtheorem{corollary}[theorem]{Corollary}

\theoremstyle{definition}

\newtheorem{example}[theorem]{Example}

\newtheorem{remark}[theorem]{Remark}

\newcommand{\cR}{\mathcal{R}}
\newcommand{\cD}{\mathcal{D}}
\newcommand{\cL}{\mathcal{L}}
\newcommand{\cO}{\mathcal{O}}
\newcommand{\cA}{\mathcal{A}}

\newcommand{\cE}{\mathcal{E}}

\newcommand{\Cstar}{\mathbb{C}^*}

\newcommand{\CC}{\mathbb{C}}

\newcommand{\QQ}{\mathbb{Q}}
\newcommand{\ZZ}{\mathbb{Z}}

\newcommand{\PP}{\mathbb{P}}
\renewcommand{\AA}{\mathbb{A}}

\DeclareMathOperator{\tail}{tail}
\DeclareMathOperator{\lcm}{lcm}

\DeclareMathOperator{\im}{im}

\DeclareMathOperator{\coker}{coker}

\DeclareMathOperator{\tdiv}{T-Div}

\DeclareMathOperator{\cl}{Cl}

\DeclareMathOperator{\Cox}{Cox}
\DeclareMathOperator{\CoxSheaf}{\mathcal Cox}
\DeclareMathOperator{\Eff}{Eff}
\DeclareMathOperator{\Nef}{Nef}
\DeclareMathOperator{\Mov}{Mov}

\DeclareMathOperator{\spec}{Spec}

\DeclareMathOperator{\ord}{ord}
\DeclareMathOperator{\orb}{orb}

\DeclareMathOperator{\id}{id}
\DeclareMathOperator{\TV}{\mathbb{TV}}

\DeclareMathOperator{\coef}{coef}

\DeclareMathOperator{\conv}{conv}
\DeclareMathOperator{\cone}{cone}

\newcommand{\fan}{\mathcal{S}}

\newcommand{\ovl}{\overline}
\newcommand{\wt}{\widetilde}

\newcommand{\ttM}{\wt{\wt{M}}} % \tilde\tilde M 
\newcommand{\tM}{\wt{M}}       % \tilde M
\newcommand{\tN}{\wt{N}}       % \tilde N
\newcommand{\kA}{A}            % A = coker(Q)
\newcommand{\kAA}{A^*}         % A* = Hom(A,Q/Z)
\newcommand{\kP}{{\mathcal P}}
\newcommand{\kZP}{\ZZ^{\kP}\hspace{-0.2em}/\Z}   % Z^p/Z
\newcommand{\kZZP}{(\ZZ^{\kP}\hspace{-0.2em}/\Z)^*}   % Z^p/Z dual
\newcommand{\kQP}{\QQ^{\kP}/\Q}   % Q^p/Q
\newcommand{\kQQP}{(\QQ^{\kP}/\Q)^*}   % Q^p/Q dual
\newcommand{\kV}{{\mathcal V}}
\newcommand{\kR}{{\mathcal R}}
\newcommand{\kVR}{{\kV\cup\kR}}
\newcommand{\kF}{\ZZ^{\kVR}}     % Z^vr
\newcommand{\kFF}{\ZZ^{(\kVR)*}}   %   und das Dual
\newcommand{\kFQ}{\QQ^{\kVR}}    % Q^vr
\newcommand{\kpR}{\QQ^{\kR}_{\geq 0}}    % positiver Oktant in Q^r
\newcommand{\kpVR}{\QQ^{\kVR}_{\geq 0}}    % positiver Oktant in Q^vr
\newcommand{\kQ}{Q}            % Abbildung Z^vr -> Zp/Z
\newcommand{\kQQ}{Q^v}         %   und das Dual
\newcommand{\kRR}{G}         % Abbildung Z^vr -> \tilde M
\newcommand{\ks}{s}            % Schnitt s
\newcommand{\kss}{s^v}         %   und das Dual
\newcommand{\kt}{t}            % Schnitt t
\newcommand{\kL}{L}            % L = im (Q)
\newcommand{\kLL}{L^*}         %   und das Dual
\newcommand{\kphi}{\varphi}    % \phi: Z^vr -> N
\newcommand{\kphii}{\varphi^v} %   und das Dual
\newcommand{\pFan}{\mathcal{S}} % p-fan S
\newcommand{\pDiv}{\mathcal{D}} % p-Div D
\newcommand{\ktRR}{F}            % Abbildung Z^rv* nach M-doppelTilde
\renewcommand{\div}{\mbox{\rm div}}
\newcommand{\largestint}[1]{\lfloor #1 \rfloor}
\newcommand{\rounddown}[1]{\largestint{#1}}

\newcommand{\Z}{\ZZ}
\newcommand{\Q}{\QQ}
\DeclareMathOperator{\kDiv}{Div}
\newcommand{\surj}{\rightarrow\hspace{-0.8em}\rightarrow}
\newcommand{\ko}{\overline}
\newcommand{\gExt}{\mbox{\rm Ext}}

\newcommand{\gHom}{\mbox{\rm Hom}}

\newcommand{\gH}{\mbox{\rm H}}

\newcommand{\kG}{\Gamma}
\newcommand{\Tors}{\operatorname{Tors}}
\newcommand{\kst}{\,|\;}
\newcommand{\kk}{\CC}
\DeclareMathOperator{\Pol}{Pol}
\DeclareMathOperator{\Spec}{Spec}

\DeclareMathOperator{\CaDiv}{CaDiv}
\DeclareMathOperator{\Div}{Div}
\DeclareMathOperator{\Pic}{Pic}
\DeclareMathOperator{\Cl}{Cl}
\DeclareMathOperator{\innt}{int}
\newcommand{\CO}{{\mathcal O}}
\newcommand{\dual}{^{\scriptscriptstyle\vee}}
\DeclareMathOperator{\loc}{loc}
\renewcommand{\iff}{\Leftrightarrow}

\newcommand{\til}[1]{\widetilde{#1}}
\DeclareMathOperator{\codim}{codim}
\newcommand{\chQ}{/\hspace{-0.3em}/^{\mbox{\tiny ch}}} % Chow-Quotient

\newcommand{\keps}{\varepsilon}

\newcommand{\kbb}{{\scriptstyle \bullet}}
\newcommand{\rato}{-\hspace{-0.3em}\to}  % rationale Abb
\DeclareMathOperator{\Grass}{Grass}

\newcommand{\T}{\mathbb T}
\newcommand{\V}{\mathbb V}
\newcommand{\toric}[1]{\T\V({#1})}  %

\renewcommand{\ttM}{\widehat{M}}
\newgray{gray1}{0.4}
\newgray{gray2}{0.5}
\newgray{gray3}{0.6}
\newgray{gray4}{0.7}
\newgray{gray5}{0.8}
\newgray{gray6}{0.9}
\newgray{gray7}{0.3}

\newcommand{\dPfan}{%
\psset{xunit=.7cm,yunit=.5cm}%
\begin{pspicture}(-3.2,-3.2)(3.2,3.2)%

\psline{<->}(-3,-2)(3,-2)%
\psline{|-|}(-1.33,-2)(-1,-2)%
\uput*[270](-1.6,-2){\tiny{$-\frac{4}{3}$}}%
\uput*[270](-0.9,-2){\tiny{$-1$}}%
\uput*[270](2,-2.2){\tiny{$\fan_\infty$}}%

\psline{<-}(-3,2)(0.66,2)%
\psline{|->}(0.66,2)(3,2)%
\uput*[270](0.66,2){\tiny{$\frac{2}{3}$}}%
\uput*[270](2,1.8){\tiny{$\fan_0$}}%

\psline{<-}(-3,0)(0.5,0)%
\psline{|->}(0.5,0)(3,0)%
\uput*[270](0.5,0){\tiny{$\frac{1}{2}$}}%
\uput*[270](2,-0.2){\tiny{$\fan_1$}}%

\end{pspicture}}

\newcommand{\dPfanPlus}{%
\psset{xunit=.7cm,yunit=.5cm}%
\begin{pspicture}(-3.2,-3.2)(3.2,3.2)%

\psline{<-|}(-3,1)(0,1)%
\psline{->}(0,1)(3,1)%
\uput*[270](2,0.8){\tiny{$\tail \fan$}}%

\psline{<-|}(-3,-1)(-0.166,-1)%
\psline{|->}(0.166,-1)(3,-1)%
\psline[linestyle=dotted,dotsep=1pt]{-}(-0.166,-1)(0.166,-1)%
\uput*[270](-0.4,-1){\tiny{$-\frac{1}{6}$}}%
\uput*[270](0.3,-1){\tiny{$\frac{1}{6}$}}%
\uput*[270](2,-1.2){\tiny{$\deg \fan$}}%

\end{pspicture}}

\newcommand{\dPtorsionFan}{%
\psset{xunit=.7cm,yunit=.5cm}%
\begin{pspicture}(-3.2,-3.2)(3.2,3.2)%

\psline{<->}(-3,-2)(3,-2)%
\psline{|-|}(-1.5,-2)(-1,-2)%
\uput*[270](-1.8,-2){\tiny{$-\frac{3}{2}$}}%
\uput*[270](-0.9,-2){\tiny{$-1$}}%
\uput*[270](2,-2.2){\tiny{$\fan_\infty$}}%

\psline{<-}(-3,2)(0.66,2)%
\psline{|->}(0.66,2)(3,2)%
\uput*[270](0.66,2){\tiny{$\frac{2}{3}$}}%
\uput*[270](2,1.8){\tiny{$\fan_0$}}%

\psline{<-|}(-3,0)(0.66,0)%
\psline{|->}(0.66,0)(3,0)%
\uput*[270](0.66,0){\tiny{$\frac{2}{3}$}}%
\uput*[270](2,-0.2){\tiny{$\fan_1$}}%

\end{pspicture}}

\newcommand{\dPtorsionFanPlus}{%
\psset{xunit=.7cm,yunit=.5cm}%
\begin{pspicture}(-3.2,-3.2)(3.2,3.2)%

\psline{<-|}(-3,1)(0,1)%
\psline{->}(0,1)(3,1)%
\uput*[270](2,0.8){\tiny{$\tail \fan$}}%

\psline{<-|}(-3,-1)(-0.166,-1)%
\psline{|->}(0.33,-1)(3,-1)%
\psline[linestyle=dotted,dotsep=1pt]{-}(-0.166,-1)(0.33,-1)%
\uput*[270](-0.4,-1){\tiny{$-\frac{1}{6}$}}%
\uput*[270](0.5,-1){\tiny{$\frac{1}{3}$}}%
\uput*[270](2,-1.2){\tiny{$\deg \fan$}}%

\end{pspicture}}

\newcommand{\cotangfannullv}{%
\psset{xunit=0.5cm,yunit=.4cm}
\begin{pspicture}(-3.2,-3.2)(3.2,3.2)%
\pspolygon[linewidth=.001pt,linecolor=white,fillstyle=solid,fillcolor=gray1](0,3)(0,1)(2,3)%
\pspolygon[linewidth=.001pt,linecolor=white,fillstyle=solid,fillcolor=gray3](2,3)(0,1)(0,0)(3,0)%
\pspolygon[linewidth=.001pt,linecolor=white,fillstyle=solid,fillcolor=gray5](3,0)(0,0)(0,-3)%
\pspolygon[linewidth=.001pt,linecolor=white,fillstyle=solid,fillcolor=gray4](0,-3)(0,0)(-3,-3)%
\pspolygon[linewidth=.001pt,linecolor=white,fillstyle=solid,fillcolor=gray2](-3,-3)(0,0)(0,1)(-3,1)%
\pspolygon[linewidth=.001pt,linecolor=white,fillstyle=solid,fillcolor=gray6](-3,1)(0,1)(0,3)%
\psset{linewidth=1pt}%

\psline{-}(0,3)(0,1)(2,3)%
\psline{-}(2,3)(0,1)(0,0)(3,0)%
\psline{-}(3,0)(0,0)(0,-3)%
\psline{-}(0,-3)(0,0)(-2,-2)%
\psline{-}(-3, -3)(0,0)(0,1)(-3,1)%
\psline{-}(-3, 1)(0,1)(0,3)%
\end{pspicture}}

\newcommand{\cotangfaninftyv}{% 
\psset{xunit=0.5cm,yunit=.4cm}
\begin{pspicture}(-3.2,-3.2)(3.2,3.2)%
\pspolygon[linewidth=.001pt,linecolor=white,fillstyle=solid,fillcolor=gray1](0,3)(0,0)(3,3)%
\pspolygon[linewidth=.001pt,linecolor=white,fillstyle=solid,fillcolor=gray3](3,3)(0,0)(3,0)%
\pspolygon[linewidth=.001pt,linecolor=white,fillstyle=solid,fillcolor=gray5](3,0)(0,0)(-1,-1)(-1,-3)%
\pspolygon[linewidth=.001pt,linecolor=white,fillstyle=solid,fillcolor=gray4](-1,-3)(-1,-1)(-3,-3)%
\pspolygon[linewidth=.001pt,linecolor=white,fillstyle=solid,fillcolor=gray2](-3,-3)(-1,-1)(-3,-1)%
\pspolygon[linewidth=.001pt,linecolor=white,fillstyle=solid,fillcolor=gray6](-3,-1)(-1,-1)(0,0)(0,3)%

\psset{linewidth=1pt}%
\psline{-}(0,3)(0,0)(3,3)%
\psline{-}(3,3)(0,0)(3,0)%
\psline{-}(3,0)(0,0)(-1,-1)(-1,-3)%
\psline{-}(-1,-3)(-1,-1)(-2,-2)%
\psline{-}(-3, -3)(-1,-1)(-3,-1)%
\psline{-}(-3, -1)(-1,-1)(0,0)(0,3)%
\end{pspicture}}

\newcommand{\cotangfanonev}{%
\psset{xunit=0.5cm,yunit=.4cm}
\begin{pspicture}(-3.2,-3.2)(3.2,3.2)%
\pspolygon[linewidth=.001pt,linecolor=white,fillstyle=solid,fillcolor=gray1](0,3)(0,0)(1,0)(3,2)%
\pspolygon[linewidth=.001pt,linecolor=white,fillstyle=solid,fillcolor=gray3](3,2)(1,0)(3,0)%
\pspolygon[linewidth=.001pt,linecolor=white,fillstyle=solid,fillcolor=gray5](3,0)(1,0)(1,-3)%
\pspolygon[linewidth=.001pt,linecolor=white,fillstyle=solid,fillcolor=gray4](1,-3)(1,0)(0,0)(-3,-3)%
\pspolygon[linewidth=.001pt,linecolor=white,fillstyle=solid,fillcolor=gray2](-3,-3)(0,0)(-3,0)%
\pspolygon[linewidth=.001pt,linecolor=white,fillstyle=solid,fillcolor=gray6](-3,0)(0,0)(0,3)%

\psset{linewidth=1pt}%
\psline{-}(0,3)(0,0)(1,0)(3,2)%
\psline{-}(3,2)(1,0)(3,0)%
\psline{-}(3,0)(1,0)(1,-3)%
\psline{-}(1,-3)(1,0)(0,0)(-3,-3)%
\psline{-}(-3,-3)(0,0)(-3,0)%
\psline{-}(-3,0)(0,0)(0,3)%
\end{pspicture}}

\newcommand{\tailFan}{%
\psset{xunit=0.5cm,yunit=.4cm}
\begin{pspicture}(-3.2,-3.2)(3.2,3.2)%
\pspolygon[linewidth=.001pt,linecolor=white,fillstyle=solid,fillcolor=gray1](0,3)(0,0)(3,3)%
\pspolygon[linewidth=.001pt,linecolor=white,fillstyle=solid,fillcolor=gray3](3,3)(0,0)(3,0)%
\pspolygon[linewidth=.001pt,linecolor=white,fillstyle=solid,fillcolor=gray5](3,0)(0,0)(0,-3)%
\pspolygon[linewidth=.001pt,linecolor=white,fillstyle=solid,fillcolor=gray4](0,-3)(0,0)(-3,-3)%
\pspolygon[linewidth=.001pt,linecolor=white,fillstyle=solid,fillcolor=gray2](-3,-3)(0,0)(-3,0)%
\pspolygon[linewidth=.001pt,linecolor=white,fillstyle=solid,fillcolor=gray6](-3,0)(0,0)(0,3)%

\psset{linewidth=1pt}%
\psline{-}(0,3)(0,-3)%
\psline{-}(-3,0)(3,0)%
\psline{-}(-3,-3)(3,3)%
\end{pspicture}}

\newcommand{\degreeFan}{%
\psset{xunit=0.5cm,yunit=.4cm}
\begin{pspicture}(-3.2,-3.2)(3.2,3.2)%
\pspolygon[linewidth=.001pt,linecolor=white,fillstyle=solid,fillcolor=gray1](0,3)(0,1)(1,1)(3,3)%
\pspolygon[linewidth=.001pt,linecolor=white,fillstyle=solid,fillcolor=gray3](3,3)(1,1)(1,0)(3,0)%
\pspolygon[linewidth=.001pt,linecolor=white,fillstyle=solid,fillcolor=gray5](3,0)(1,0)(0,-1)(0,-3)%
\pspolygon[linewidth=.001pt,linecolor=white,fillstyle=solid,fillcolor=gray4](0,-3)(0,-1)(-1,-1)(-3,-3)%
\pspolygon[linewidth=.001pt,linecolor=white,fillstyle=solid,fillcolor=gray2](-3,-3)(-1,-1)(-1,0)(-3,0)%
\pspolygon[linewidth=.001pt,linecolor=white,fillstyle=solid,fillcolor=gray6](-3,0)(-1,0)(0,1)(0,3)%
\pspolygon[linewidth=.001pt,linecolor=white,fillstyle=crosshatch,hatchsep=0.1,hatchangle=0](0,1)(1,1)(1,0)(0,-1)(-1,-1)(-1,0)%

\psset{linewidth=1pt}%
\psline{-}(0,1)(1,1)(1,0)(0,-1)(-1,-1)(-1,0)(0,1)%
\psline{-}(0,1)(0,3)
\psline{-}(1,1)(3,3)%
\psline{-}(1,0)(3,0)%
\psline{-}(0,-1)(0,-3)%
\psline{-}(-1,-1)(-3,-3)%
\psline{-}(-1,0)(-3,0)%
\end{pspicture}}

\newcommand{\degenerationZero}{%
\psset{xunit=0.5cm,yunit=.4cm}
\begin{pspicture}(-3.2,-3.2)(3.2,3.2)%
\pspolygon[linewidth=.001pt,linecolor=white,fillstyle=solid,fillcolor=gray1](0,3)(0,1)(1,1)(3,3)%
\pspolygon[linewidth=.001pt,linecolor=white,fillstyle=solid,fillcolor=gray3](3,3)(1,1)(1,0)(3,0)%
\pspolygon[linewidth=.001pt,linecolor=white,fillstyle=solid,fillcolor=gray5](3,0)(1,0)(0,-1)(0,-3)%
\pspolygon[linewidth=.001pt,linecolor=white,fillstyle=solid,fillcolor=gray4](0,-3)(0,-1)(-1,-1)(-3,-3)%
\pspolygon[linewidth=.001pt,linecolor=white,fillstyle=solid,fillcolor=gray2](-3,-3)(-1,-1)(-1,0)(-3,0)%
\pspolygon[linewidth=.001pt,linecolor=white,fillstyle=solid,fillcolor=gray6](-3,0)(-1,0)(0,1)(0,3)%
\pspolygon[linewidth=.001pt,linecolor=white,fillstyle=solid,fillcolor=gray7](0,1)(1,1)(1,0)(0,-1)(-1,-1)(-1,0)%

\psset{linewidth=1pt}%
\psline{-}(0,1)(1,1)(1,0)(0,-1)(-1,-1)(-1,0)(0,1)%
\psline{-}(0,1)(0,3)
\psline{-}(1,1)(3,3)%
\psline{-}(1,0)(3,0)%
\psline{-}(0,-1)(0,-3)%
\psline{-}(-1,-1)(-3,-3)%
\psline{-}(-1,0)(-3,0)%
\end{pspicture}}

\begin{document}

\title{Cox Rings of Rational Complexity One $T$-Varieties}

\author[K.~Altmann]{Klaus Altmann}
\address{Institut f\"ur Mathematik und Informatik,
        Freie Universit\"at Berlin,
        Arnimallee 3,
        14195 Berlin, Germany}
\email{altmann@math.fu-berlin.de}

\author[L.~Petersen]{Lars Petersen}
\address{Institut f\"ur Mathematik und Informatik,
         Freie Universit\"at Berlin
	 Arnimallee 3,
	 14195 Berlin, Germany}
\email{petersen@math.fu-berlin.de}


\begin{abstract}
Let $X$ be a Mori dream space together with an effective torus action of complexity one. In this note, we construct a polyhedral divisor $\cD_{\Cox}$ on a suitable finite covering of $\PP^1$ which corresponds to the Cox ring of $X$. This description allows for a detailed study of torus orbits and deformations of $\Cox(X)$. Moreover, we present coverings of $\PP^1$ together with an action of a finite abelian group $A$ in terms of so-called $A$-divisors of degree zero on $\PP^1$. 
\end{abstract}

\maketitle

%%%%%%%%%%%%%%%%%%%%%%%%%%%%%%%%%%%%%%%%%%%%%%%%%%%%%%%%%%%%%%%%%%%%%
%%%%%%%%%%%%%%%%%%%%%%%%%%%%%%%%%%%%%%%%%%%%%%%%%%%%%%%%%%%%%%%%%%%%%
%%%
%%%   Introduction
%%%
%%%%%%%%%%%%%%%%%%%%%%%%%%%%%%%%%%%%%%%%%%%%%%%%%%%%%%%%%%%%%%%%%%%%%
%%%%%%%%%%%%%%%%%%%%%%%%%%%%%%%%%%%%%%%%%%%%%%%%%%%%%%%%%%%%%%%%%%%%%

\section{Introduction}
\label{intro}

For simplicity,
all varieties appearing in this article are supposed to be complex.

%%%%%%%%
% MDS
%%%%%%%%

\subsection{Mori dream spaces}
\label{MDS}

Let $X$ be a normal, $\Q$-factorial, complete variety
with a finitely generated divisor class group $\Cl(X)$.
Thus, $\Cl_\Q(X):= \Cl(X)\otimes_\Z\Q$
coincides with $\Pic_\Q(X)$ as well as with the N\'{e}ron-Severi group $N_\Q^1(X)$. This setting gives rise to the definition
of the $\Cl(X)$-graded abelian group
$$
\Cox(X):=\bigoplus_{D\in\Cl(X)}\kG(X,\CO_X(D)).
$$
Moreover, $\Cox(X)$ carries a canonical ring structure
(the ``Cox ring'' of $X$): If
$\Cl(X)$ is torsion free, then this can be achieved by taking the divisors $D$
from a fixed section $\Cl(X)\hookrightarrow\Div(X)$ of the natural surjection
$\Div(X)\surj \Cl(X)$. If $\Cl(X)$ has torsion, then one has to work
with a presentation of $\Cl(X)$ by a finitely generated
subgroup of $\Div(X)$
with relations instead. See \cite[Sect.2]{tvarcox} for a detailed treatment
of this case.%
\\[1ex]
By \cite{MDS}, we call $X$ a Mori dream space (MDS) if $\Cox(X)$ is a finitely
generated $\CC$-algebra. This property has two important consequences.
First, every nef divisor is semi-ample, i.e.\ the latter property can be
checked numerically. Second,
since the Cox ring carries the information about
all rational maps of $X$, the data of the minimal model
program (MMP) are finite: Not only are the ample, movable, and effective 
cones $\Nef(X)\subseteq\Mov(X)\subseteq\Eff(X)\subseteq N^1_\Q(X)$ polyhedral,
but $\Eff(X)$ carries a finite polyhedral subdivision such that the 
birational transformations $X_i$
of $X$ correspond to the cells of this subdivision.
Actually, the $X_i$ appear as GIT quotients of the total coordinate space
$\Spec\Cox(X)$, and the polyhedral subdivision of $\Eff(X)$ corresponds to
the GIT equivalence classes.

%%%%%%%%%%%
% toricCox
%%%%%%%%%%%

\subsection{The toric case}
\label{toricCox}

Probably the most well known example of an MDS is the class of toric varieties.
Let $M,N$ be two mutually dual, finitely generated, free abelian groups,
and denote by $M_\Q, N_\Q$ the associated rational vector spaces. 
For a fan $\Sigma$ in $N_\Q$ we denote by $\toric{\Sigma}$ the associated toric variety.
It was shown in \cite{cox} that $\Cox(\toric{\Sigma})$ is a polynomial ring whose variables correspond to the rays of $\Sigma$.
This fact, by the way, characterizes toric varieties.
\\[1ex]
While this describes $\Cox(\toric{\Sigma})$ completely in algebraic terms,
we would like to emphasize a slightly alternative,
more polyhedral point of view:
By abuse of notation, we denote the first non-trivial lattice point on
a ray $\varrho\in\Sigma(1)\subseteq N$ with the same letter~$\varrho$.
Then we consider the
canonical map $\varphi:\Z^{\Sigma(1)}\to N$, $e(\varrho) \mapsto \varrho$.
It sends some faces (including the rays)
of the positive orthant $\Q^{\Sigma(1)}_{\geq 0}$
to cones of the fan $\Sigma$. Applying the functor $\toric{\kbb}$, we obtain
a rational map
$\Spec \CC[z_\varrho\kst \varrho\in\Sigma(1)]\rato \toric{\Sigma}$.
In particular, we recover the affine spectrum of
$\CC[z_\varrho\kst \varrho\in\Sigma(1)]=\Cox(\toric{\Sigma})$
as the toric variety $\toric{\Q^{\Sigma(1)}_{\geq 0}}$.
Thus, the Cox ring of a toric variety gives rise to an affine toric variety
itself,
and the defining cone $\Q^{\Sigma(1)}_{\geq 0}$ can be seen as a 
polyhedral resolution of the given fan $\Sigma$, since all linear relations among the rays have been removed.

%%%%%%%%
% jarek
%%%%%%%%

\subsection{Action of the Picard torus}
\label{jarek}

In \cite{tvar_1} the concept of affine toric varieties was generalized
to describe normal, $n$-dimensional,
affine varieties $Z$ with an effective action
of a torus $T\cong (\CC^*)^k$.
The combinatorial part of the language is based on the character lattice 
$M:=\gHom(T,\CC^*)\cong\Z^k$ and its dual $N$. 
While toric varieties 
as in (\ref{toricCox}) (the case $k=n$) are determined by cones and fans
inside $N_\Q$, the general situation ($k\leq n$)
must also involve an $(n-k)$-dimensional
geometric part, i.e.\ some quasiprojective variety $Y$.
If $Z$ is given, then $Y=Z\chQ T$ will be the so-called Chow quotient.
Both parts are combined in the notion of a p-divisor
$\pDiv=\sum_i\Delta_i\otimes D_i$ on $Y$, i.e.\ a divisor on $Y$ where the
coefficients $\Delta_i$ are polyhedra in $N_\Q$. 
Roughly speaking, \cite{tvar_1} establishes a one-to-one correspondence
$Z$ $\leftrightarrow$ $(N,Y,\pDiv)$.
See (\ref{pDivDef}) for the details.
\\[1ex]
If, as in (\ref{MDS}), $X$ is an MDS, and if we suppose that
$\Cl(X)$ is torsion free,
then $Z:=\Spec\Cox(X)$ is a normal affine variety, and the
$\Cl(X)$-grading encodes an effective action of the so-called Picard torus
$T:=\gHom(\Cl(X),\CC^*)$. Thus, one could ask for a description of
$Z$ in terms of some p-divisor $\pDiv$ on some $Y$.
This has been done in \cite{coxmds}, and in the case of smooth MDS surfaces,
the result is as follows:
$Y=X$ and, up to shifts of the polyhedral coefficients,
$\pDiv=\sum_{E\subseteq X}\Delta_E\otimes E$ with
$$
\Delta_E=\{D\in\Eff(X)\subseteq\Cl_\Q(X)\kst
(D\cdot E)\geq -1 \mbox{ and } (D\cdot E')\geq 0 \mbox{ for $E'\neq E$}\}
$$ 
where $E,E'$ run through all negative curves in $X$.
One easily sees that the common tailcone, cf.\ (\ref{pDivDef}),
of the $\Delta_E$ equals $\Nef(X)$ which is dual to $\Eff(X)$. Note that the latter cone carries the
degrees of $\Cox(X)$. In the case that $X$ is a del Pezzo surface, the formula for $\Delta_E$ simplifies to
$\Delta_E= \ko{0 E} + \Nef(X)\subseteq\Cl_\Q(X)$.

%%%%%%%%%%%%%%%%%%%%%%
% c1JH (Hausen/Suess)
%%%%%%%%%%%%%%%%%%%%%%

\subsection{Complexity one}
\label{c1JH}

We now return to the case where our originally given variety $X$ 
comes with an effective action of a torus $T$.
In contrast to (\ref{toricCox}), we assume that
this action is of complexity one, i.e.\ the Chow quotient
$Y:=X\chQ T$ is a curve. For $X$ to be an MDS, we must have
$Y=\PP^1$, and such an $X$
is always representable as $X(\pFan)$ for a so-called
divisorial fan $\pFan=\sum_{p\in\PP^1}\pFan_p\otimes [p]$.
This notion will be recalled in (\ref{non-aff}), but
it is very similar to the notion of a p-divisor previously
used in (\ref{jarek}).
The points $p\in\PP^1$ yield divisors $[p]$, and
since $X$ is no longer affine, the polyhedral coefficients $\Delta_p$
have been replaced by polyhedral subdivisions $\pFan_p$, so-called slices, of the vector space
$N_\Q$ with $N:=\gHom(\CC^*,T)$.
\\[1ex]
In \cite{tvarcox}, the generators and relations of the
Cox ring of $X(\pFan)$ were calculated in terms of $\pFan$.
More precisely, if $\kR$ denotes the set of rays of $\tail\pFan$ being disjoint to $\deg\pFan$, cf.\ (\ref{globX}), and if $p\in\PP^1$ runs through the points with non-trivial slices $\pFan_p$, then
\cite[Corollary 4.9]{tvarcox} says that
$$
\textstyle
\Cox(X)\;=\;\CC[S_\varrho\kst \varrho\in\kR]\;\otimes_{\CC}\;
\CC[T_{v,p}\kst v\in\pFan_p(0)]\big/
\sum_p\beta_p \prod_{v\in\pFan_p(0)} T_{v,p}^{\mu(v)}
$$
where $\beta_\kbb$ runs through all (or a basis of the) linear relations
$\sum_p\beta_p\,\til{p}=0$ among some fixed lifts 
$\til{p}\in\CC^2\setminus\{0\}$ of the points $p\in\PP^1$
and $\mu(v)$ denotes the denominator of the vertices $v\in\pFan_p(0)\subseteq N_\Q$.

%%%%%%%%%%%%%%%%%%%
% Example ex-Omega
%%%%%%%%%%%%%%%%%%%
\begin{example}
\label{ex-Omega}
\cite[Example 4.4]{tvarcox}.
The three slices $\fan_0,\fan_1,\fan_{\infty}$ (see Figure~\ref{fig:cotangFan}) for $p=0,1,\infty \in\PP^1$
with tailfan $\Sigma$ and degree $\deg \pFan$ (see Figure~\ref{fig:cotangPlus}) encode the projectivized cotangent bundle $X=\PP(\Omega_{\PP^2})$
over $\PP^2$. The latter is a toric variety, and its torus $T\subseteq\PP^2$ still acts on $X$.

\begin{figure}[htbp]
  \centering
  \subfigure[$\fan_0$]{\cotangfannullv}
  \subfigure[$\fan_1$]{\cotangfanonev}
  \subfigure[$\fan_\infty$]{\cotangfaninftyv}
  \caption{Fansy divisor of $\PP(\Omega_{\PP^2})$.}
  \label{fig:cotangFan}
\end{figure}

In this example, we have $\kR=\emptyset$, and 
there is exactly one relation among
$\til{p}=(1,0), (1,1), (0,1)\in\CC^2$, namely $\beta=(1,-1,1)$.
Since all vertices $v_1,v_2\in\pFan_0(0)$,
$v_3,v_4\in\pFan_1(0)$, and  $v_5,v_6\in\pFan_{\infty}(0)$
are lattice points, we end up with
$$
\textstyle
\Cox(X)\;=\;\CC[T_{1},\ldots,T_6]\big/
(T_1T_2-T_3T_4+T_5T_6).
$$

\begin{figure}[htbp]
  \centering
  \subfigure[$\Sigma = \tail \fan$]{\tailFan}
  \subfigure[$\deg \fan$]{\degreeFan}
  \caption{Tailfan and degree of $\fan$.}
  \label{fig:cotangPlus}
\end{figure}
\end{example}

%%%%%%%%
% c1KL
%%%%%%%%

\subsection{Cox rings as p-divisors}
\label{c1KL}

Our approach to the Cox ring
is somewhat different. Keeping the setting from (\ref{c1JH}), we will describe $\Cox(X)$ (or rather its affine spectrum)
as a p-divisor. So far, this is similar to the viewpoint of \cite{coxmds}
described in (\ref{jarek}). However, the special feature of the present paper
is that the given $T$-action on $X$ is bequeathed to $\Cox(X)$, and by combining it with the action of the Picard torus, $\Spec \Cox(X)$ turns into a complexity one variety, too. Our goal is to stay within the language of \cite{tvar_1} and to present $\Cox (X)$ as a p-divisor $\cD_{\Cox}$ on a curve $C$.
\\[1ex]
The combination of the two torus actions might in general involve
torsion. Thus, one could try to understand $\Spec \Cox(X)$ as a
variety with the complexity one action of a diagonizable group.
Then, the Chow quotient $Y$ would be the same $\PP^1$ carrying the divisorial fan $\pFan$. However, since there is no general theory of p-divisors
for those groups yet, we have to divide out the torsion.
Our main result is the description of the now well-defined p-divisor
$\pDiv_{\Cox}$ in Theorem~\ref{th-mainTh}.
Removing the torsion gives rise to a finite covering $C \to \PP^1$, so that our polyhedral divisor $\pDiv_{\Cox}$ will live on a curve $C$ of probably higher genus rather than on $\PP^1$.
\\[1ex]
Yet, if the class group $\cl(X)$ is torsion free, then this will not happen, i.e.\ $C \to \PP^1$ is the identity map. In this case, the p-divisor $\pDiv_{\Cox}$ describing
$\Spec \Cox(X)$ utilizes the very same points $p\in\PP^1$ 
as $\pFan$:
Denoting by $p \in \kP \subseteq \PP^1$ the set of points with non-trivial 
slices $\pFan_p$ and by $\kV:=\bigcup_{p\in\kP}\pFan_p(0)$ the corresponding vertices, we define the compact polytopes
$$
\Delta_p^c := \conv\{e(v)/\mu(v)\kst v\in \pFan_p(0)\}
\subseteq\kFQ
$$
where the $e(v)\in\Z^{\kV}$ denote the canonical basis vectors and
$\mu(v)\in\Z_{\geq 1}$ is again the smallest positive integer turning
$\mu(v)v$ into a lattice point, i.e.\ transfers it from $N_\Q$ into a
genuine element of $N$, cf.\ (\ref{funkX}).
This gives rise to the polyhedral cone
$$
\sigma \;:=\; \Q_{\geq 0}\cdot\prod_{p\in\kP} \Delta_p^c 
\;+\; \kpR \;\subseteq\; \kpVR.
$$

%%%%%%%%%%%%%%%%%%%%%%%%%
% theorem th-mainThIntro
%%%%%%%%%%%%%%%%%%%%%%%%%
\begin{theorem}
\label{th-mainThIntro}
If $\cl(X)$ is torsion free, then the p-divisor $\cD_{\Cox}$ of
$\Spec \Cox(X)$ is, up to shifts of the polyhedral coefficients,
given by $\sum_{p\in\kP}(\Delta_p^c+\sigma)\otimes [p]$
on $\PP^1$. 
\end{theorem}

For the details concerning the shifts as well as for the general result that also covers the case of a class group with torsion, see
Theorem~\ref{th-mainTh}.
Comparing this with the toric case of (\ref{toricCox}),
we observe a similar construction of resolving linear dependencies
among elements of $N_\Q$ by reserving separate dimensions for each
of them. The Cox ring of $T$-varieties of complexity one is therefore closely related to a fiberwise toric Cox construction with respect to the rational map $X\rato\PP^1$.

%%%%%%%%%%%%%%%%%%%
% Example ex-Omega
%%%%%%%%%%%%%%%%%%%
\begin{example}
\label{ex:coxDivisorCotang}
% \cite[p.216]{tvar_2} and 
Let $\pFan$ be the divisorial fan of Example~\ref{ex-Omega}. The three polytopes $\Delta_p^c$ are compact edges,
and $\sigma$ becomes a four-dimensional cone over a cube.
Hence, according to Example~\ref{ex-Omega},
the resulting p-divisor is that of the affine cone over $\Grass(2,4)$
from \cite[p.\,849]{fansy}.
\end{example}

%%%%%%%%
% plan
%%%%%%%%
\subsection{}
\label{plan}
The paper is organized as follows:
In Section~\ref{recallPD} we recall 
the notions of p-divisors and divisorial fans.
Section~\ref{finiteCov} is of independent interest and contains the 
necessary tools for the description of finite smooth coverings of $\PP^1$. The language was chosen to be as close as possible to that of p-divisors. This means that we present a covering of $\PP^1$ together with an action of a finite abelian group $A$ in terms of an $A$-divisor of degree zero on $\PP^1$, cf.\ Theorem \ref{th-smoothness}.
Section~\ref{pDivCox} presents our main result
together with its proof, the latter beginning with a diagram comprising all necessary data (\ref{torsionDiagram}). We conclude the paper with a couple of examples in Section~\ref{ex}.

%%%%%%%%%%%%%%%%%%%%%%%%%%%%%%%%%%%%%%%%%%%%%%%%%%%%%%%%%%%%%%%%%%%%%
%%%%%%%%%%%%%%%%%%%%%%%%%%%%%%%%%%%%%%%%%%%%%%%%%%%%%%%%%%%%%%%%%%%%%
%%%
%%%   Polyhedral divisors
%%%
%%%%%%%%%%%%%%%%%%%%%%%%%%%%%%%%%%%%%%%%%%%%%%%%%%%%%%%%%%%%%%%%%%%%%
%%%%%%%%%%%%%%%%%%%%%%%%%%%%%%%%%%%%%%%%%%%%%%%%%%%%%%%%%%%%%%%%%%%%%
\section{Polyhedral Divisors}
\label{recallPD}

%%%%%%%%%%
% pDivDef
%%%%%%%%%%

\subsection{Definition of p-divisors}
\label{pDivDef}

First, we recall the basic notions of \cite{tvar_1}. Let $T$ be a
$k$-dimensional affine torus ($\cong(\CC^*)^k$), and let $M$, $N$ denote the mutually dual, free abelian 
groups ($\cong\Z^k$) of characters and one-parameter subgroups, 
respectively.
In particular, $T$ can be recovered as $T=\Spec\CC[M]=N\otimes_\Z\CC^*$.
Denote by $M_\Q$, $N_\Q$ the corresponding vector spaces over $\Q$.
Then, for a polyhedral cone $\sigma\subseteq N_\Q$,
we may consider the semigroup (with respect to Minkowski addition)
$$ 
\Pol^+_\Q(N,\sigma):=\{\Delta\subseteq N_\Q\kst
\Delta=\mbox{polyhedron with} \,\tail\Delta=\sigma\}
\subseteq \Pol_\Q(N,\sigma)
$$
where $\tail\Delta:=\{a\in N_\Q\kst \Delta+a\subseteq \Delta\}$
denotes the {\em tailcone} of $\Delta$
and $\Pol\supseteq\Pol^+$ is the associated Grothendieck group.
\\[1ex]
On the other hand, let $Y$ be a normal and semiprojective 
(i.e.\ $Y\to Y_0$ is projective over an affine $Y_0$) variety.
By $\CaDiv(Y)$ we denote the group of Cartier divisors on $Y$.
A $\Q$-Cartier divisor on $Y$ is called {\em semiample} if it has
a positive, base point free multiple.
For an element 
$$
\pDiv = \sum_i \Delta_i\otimes D_i\in 
\Pol_\Q(N,\sigma)\otimes_\Z \CaDiv(Y)
$$ 
with $\Delta_i\in\Pol^+(N_\Q,\sigma)$
and effective divisors $D_i$, we may consider its evaluations
$$
\pDiv(u) := \sum_i \min\langle \Delta_i,u\rangle D_i\in \CaDiv_\Q(Y)
$$
on elements $u\in\sigma\dual$. We call $\pDiv$ a polyhedral or 
{\em p-divisor} if the $\pDiv(u)$ are semiample and, moreover,
big for $u\in\innt\sigma\dual$.
The common tailcone $\sigma$ of the coefficients $\Delta_i$
will be denoted by $\tail(\pDiv)$.
The positivity assumptions imply that
$\pDiv(u)+\pDiv(u')\leq \pDiv(u+u')$, so that
$\CO_Y(\pDiv):= \bigoplus_{u\in\sigma^\vee\cap M} \CO_Y(\pDiv(u))$
becomes a sheaf of rings. We define $\til{X}(\pDiv):= \Spec_Y \CO_Y(\pDiv)$ over $Y$, and we furthermore set $X(\pDiv):= \Spec \kG(Y,\CO(\pDiv))$.
\\[1ex]
The latter space does not change if $\pDiv$ is pulled back via a
birational modification $Y'\to Y$ or if $\pDiv$ is modified by a
polyhedral {\em principal} divisor on $Y$ where the latter denotes an element in the image of the natural map
$N\otimes_\Z\kk(Y)^*\to \Pol_\Q(N,\sigma)\otimes_\Z \CaDiv(Y)$.
Two p-divisors that differ by chains of
the upper operations are called equivalent.
Note that this implies that one can always ask for $Y$ to be smooth.

\begin{theorem}[\cite{tvar_1}, Theorems (3.1), (3.4); Corollary (8.12)]
\label{thm-equivPT}
The map $\pDiv \mapsto X(\pDiv)$ yields a bijection between
equivalence classes of p-divisors and normal, affine varieties with an effective $T$-action.
\end{theorem}

Finally, since we always want to assume that $Y$ is projective, we will explicitly allow $\emptyset\in\Pol^+_\Q(N,\sigma)$ as polyhedral coefficients of $\pDiv$. Then, $\pDiv=\sum_i \Delta_i\otimes D_i$ should be interpreted as $\sum_{\Delta_i\neq \emptyset} \Delta_i\otimes D_i|_{\loc\pDiv}$
with $\loc\pDiv:=Y\setminus \cup_{\Delta_i= \emptyset} D_i$.

%%%%%%%%%%%%%
% non-aff
%%%%%%%%%%%%%

\subsection{Complexity one}
\label{non-aff}

The situation in fact simplifies a lot in complexity one $(k=n-1)$, because $Y$ becomes a smooth projective curve.
For example, introducing the addition rule $\Delta+\emptyset:=\emptyset$, we may define the degree of
$\pDiv = \sum_i \Delta_i\otimes D_i$ as 
$$
\deg \pDiv := \sum_i (\deg D_i)\cdot \Delta_i\in \Pol^+_\Q(N,\sigma).
$$
In particular, $\deg \pDiv=\emptyset \iff \loc\pDiv\neq Y
\iff \loc \pDiv$ is affine.
This easily implies the following criterion: 
$\pDiv$ is a p-divisor if and only if 
$\deg\pDiv\subsetneq\tail\pDiv$
and, additionally,
$\pDiv((\gg 0)\cdot w)$ is principal for $w\in (\tail\pDiv)\dual$
with $w^\bot\cap(\deg\pDiv)\neq 0$. Note that the latter condition
is automatically fulfilled if $Y=\PP^1$.
\\[1ex]
The assignment $\pDiv\mapsto X(\pDiv)$ of (\ref{pDivDef})
is functorial. In particular, as was shown in \cite{NathHend}, if $\pDiv$ is a p-divisor containing some $\pDiv'=\sum_i \Delta_i'\otimes D_i$ (meaning that $\Delta'_i\subseteq\Delta_i$ for all $i$),
then $\pDiv'$ is again a p-divisor which induces a $T$-equivariant open embedding
$X(\pDiv')\hookrightarrow X(\pDiv)$
if and only if $\pDiv'\leq\pDiv$, i.e.\ if all the coefficients
$\Delta_i'\leq\Delta_i$ are faces,
and 
$$
\deg\pDiv'=\deg\pDiv\cap\tail\pDiv'.
$$
% \\[1ex]
%
In particular, if p-divisors $\pDiv^\nu$ are arranged in a so-called
divisorial fan $\pFan=\{\pDiv^\nu\}$, then we can
glue the associated affine $T$-varieties $X(\pDiv^\nu)$ to obtain a separated
$X(\pFan)=\bigcup_\nu X(\pDiv^\nu)$,
cf.\ \cite{tvar_2}.
The so-called {\em slices}
$\pFan_i=\{\Delta^\nu_i\}$ form a polyhedral
subdivision in $N_\Q$, and $\tail\pFan:=\{\tail \pDiv_i\}$ is called the tail
fan of $\pFan$. 
Moreover, the subsets $\deg \pDiv \subsetneq \tail\pDiv$ glue
to a subset $\deg\pFan\subsetneq|\tail\pFan|\subseteq N_\Q$.
Roughly speaking, we understand that
$\pFan=\sum_i \pFan_i\otimes D_i$. Yet, to keep the full information
of the divisorial fan $\pFan$ one needs a labeling of the $\pFan_i$-cells indicating the p-divisor they come from. 
\\[1ex]
However, following \cite{NathHend}, the technical description can be reduced considerably since one may eventually forget about the labeling. Instead, one only needs to mark those cones $\tail \pDiv_i$ inside $\tail\pFan$ which have non-empty $\deg \pDiv_i$. The marked cones together with the formal sum $\pFan=\sum_i \pFan_i\otimes D_i$ then give the so-called \emph{marked fansy divisor} associated with the divisorial fan $\pFan$. Conversely, given a so-called {\em fansy divisor} $\pFan$, i.e.\ a formal sum $\pFan=\sum_i \pFan_i\otimes D_i$ of polyhedral subdivisions $\pFan_i$ with common tailfan, one can find a small system of axioms for a marking of the cones of $\tail\pFan$ to yield a divisorial fan $\pFan = \{\cD^\nu\}$, cf. \cite{NathHend}.

%%%%%%%%%%%%%%%%%%%%%%%%%%%%
% Equivariant Weil divisors
%%%%%%%%%%%%%%%%%%%%%%%%%%%%

\subsection{Equivariant Weil divisors}
\label{DivX}
Let $\pDiv$ be a p-divisor of arbitrary complexity.
From \cite{tvar_1} we know that, similar to the toric case,
there is a decomposition of $\til{X}(\pDiv)$
into $T$-orbits $\orb(y,F)$ where $y\in Y$, and $F$ is a 
non-empty face of the polyhedron $\Delta_y:=\sum_{D_i\ni y}\Delta_i$
(the neutral element $\tail\pDiv$ serves as the 
sum of the empty set of summands).
The dimension of their closures is 
$\dim\, \ko{\orb}(y,F)=\dim\ko{y}+\codim_N F$.
Note that different $T$-orbits might be identified 
via the contraction $\til{X}(\pDiv)\to X(\pDiv)$.
\\[1ex]
In particular, the set of $T$-equivariant
prime divisors in $\til{X}(\pDiv)$ is twofold: On the one hand, we have the so-called {\em vertical} divisors
$$
D_{Z,v}:=\ko{\orb}(\eta(Z),v).
$$ 
They are associated with prime divisors $Z\subseteq Y$ (with $\eta(Z)$ being the generic point of $Z$) together with a vertex $v\in\Delta_Z$. On the other hand, there are the so-called {\em horizontal} divisors
$$
D_\varrho:= \ko{\orb}(\eta(Y),\varrho).
$$
These correspond to rays $\varrho$ of the cone $\tail\pDiv=\Delta_{\eta(Y)}$, where $\eta(Y)$ denotes the generic point of $Y$. The $T$-equivariant prime divisors in $X(\pDiv)$
correspond exactly to those on $\til{X}(\pDiv)$ that are not
contracted via $\til{X}(\pDiv)\to X(\pDiv)$.%
\\[1ex]
Let us now assume that $\pDiv$ is of complexity one. Then,
the vertical divisors $D_{(p,v)}$ (with $p\in Y(\CC)$ and
$v\in \Delta_p$) survive completely in $X(\pDiv)$.
In contrast, $D_\varrho$ becomes contracted if and only if
the ray $\varrho$ is not disjoint from $\deg\pDiv$.
Finally, we denote by $\,\tdiv X(\pDiv)$ the free abelian group of the $T$-equivariant Weil divisors in $X(\pDiv)$.

%%%%%%%%%%%%%%%%%%%%%%%%%%%%%%%%%
% Equivariant principal divisors
%%%%%%%%%%%%%%%%%%%%%%%%%%%%%%%%%

\subsection{Equivariant principal divisors}
\label{funkX}
Again, let $\cD$ be a p-divisor of arbitrary complexity, and let $K(Y)$ be the function field of $Y$. Then we have that
$\CO_Y(\pDiv)\subseteq K(Y)[M]\subseteq K(\til{X})=K(X)$.
In particular, $K(Y)[M]$ consists exactly of the semi-invariant,
i.e.\ $M$-homogeneous rational functions on $X$ -- and this
remains true for the non-affine $T$-varieties $X(\pFan)$ obtained by gluing.
\\[1ex]
Identifying a ray $\varrho$ with its primitive generating lattice vector and denoting by $\mu(v)$ the smallest integer $k \geq 1$ for $v\in N_\Q$
such that $k\cdot v$ is a lattice point, one has the following characterization of $T$-equivariant principal divisors on $\til{X}$ or $X$:

\begin{theorem}[\cite{tidiv}, Section 3]
\label{thm-principDiv}
Let $f(y)\chi^u\in K(Y)[M]$. The associated principal divisor 
on $\til{X}$ or $X$ is then given by
$$
\div\big(f(y)\chi^u\big)=
\sum_{\varrho} \langle\varrho,u\rangle D_\varrho +
\sum_{(Z,v)}\mu(v)\big(\langle v,u\rangle + \ord_Z f\big) \cdot D_{(Z,v)}
\vspace{-0.3ex}
$$
where, if focused on $X$, one is supposed to omit all prime divisors being contracted via $\til{X}\to X$.
\end{theorem}

Let us present a short proof that differs from the original one:

\begin{proof}
It suffices to check the formula on $\til{X}(\pDiv)$. Then everything becomes local, and we can use a formal neighborhood of $\eta(Z)\in Y$.
Thus, $(Y,\eta(Z))=(\AA^1_K,0)$ with $K=K(\eta(Z))$ being the residue field.
\\[0.5ex]
In particular, $\pDiv=\Delta\otimes[0\in\AA^1]$, which is the downgrade of a toric situation. Using \cite[\S 11]{tvar_1},
one sees that
$\wt{X}=\TV\big(\cone(\Delta,1)\subseteq N_\QQ\oplus\QQ\big)$.
Moreover, the ray $\varrho$ corresponds to $(\varrho,0)\in N\oplus\Z$,
and the vertex $v\in\Delta$ turns into the ray generated by $(v,1)$
or, equivalently, by the primitive lattice generator 
$(\mu(v)v,\mu(v))\in N\oplus\Z$.%
\\[0.8ex]
On the other hand, $f(y)\chi^u$ translates into 
$t^{\ord_Zf}\cdot\chi^u$, i.e.\ into $[u,\ord_Zf]\in M\oplus\Z$.
Pairing the latter element with the upper pairs $(\rho,0),(\mu(v)v,\mu(v)) \in N \oplus \Z$ completes the proof.
\end{proof}

%%%%%%%%%%%%%%%%%%%%%%%%%%%%%%%%%%%%%%%%
%  Divisor classes and global sections
%%%%%%%%%%%%%%%%%%%%%%%%%%%%%%%%%%%%%%%%

\subsection{Divisor classes and global sections}
\label{globX}
Let $X=X(\pFan)$ for a complete divisorial fan
$\pFan=\sum_{p\in\PP^1} \pFan_p \otimes [p]$ on $Y=\PP^1$.
In particular,
$\deg\pFan\subsetneq|\tail\pFan|= N_\Q$.
Choose a finite set of points $\kP\subseteq\PP^1$ such that 
$\pFan_p$ is trivial ($\pFan_p=\tail\pFan$)
for $p\in\PP^1\setminus\kP$.
For a vertex $v$ of some slice of $\pFan$
we denote
by $p(v)\in\PP^1$ the point whose slice $\pFan_{p(v)}$ we have taken $v$ from. Let us define the following sets:
$$
\kV\;:=\; \{v\in \pFan_{p}(0)\kst p\in\kP\}
\hspace{1em}\mbox{and}\hspace{1em}
\kR\;:=\;\{\varrho\in(\tail\pFan) (1)\kst \varrho\cap\deg\pFan=\emptyset\}.
$$
Using the notation of (\ref{funkX}),
this leads to the definition of the following two natural maps
$$
\kQ:\kF\to\kZP\hspace{1.6em}\mbox{with}\hspace{0.7em}
e(v)\mapsto \mu(v)\,\ko{e(p(v))} 
\hspace{0.7em}\mbox{and}\hspace{0.7em}
e(\varrho)\mapsto 0\,,
\vspace{-1ex}
$$
and
\vspace{-1ex}
$$
\kphi:\kF\to N\hspace{1.6em}\mbox{with}\hspace{0.7em}
e(v)\mapsto \mu(v)v\hspace{0.7em}\mbox{and}\hspace{0.7em}
e(\varrho)\mapsto \varrho
$$
with $e(v)$ and $e(\varrho)$ denoting the natural basis vectors 
as in (\ref{c1KL}).
Since we know from (\ref{DivX}) that $\kFF\subseteq\tdiv X$,
the above Theorem~\ref{thm-principDiv} now yields a direct 
description of the class group $\cl(X)$.

\begin{corollary}
\label{cor-ClXP1}
For $X=X(\pFan)$, one has the exact sequence
$$
0\to\kZZP\oplus M \to \kFF\to \cl(X)\to 0
$$
where the first map is induced from $(\kQ,\kphi)$.
\end{corollary}

\begin{proof}
If we dealt with the whole projective line $\PP^1$ instead of
the finite subset $\kP$, then $(\Z^{\PP^1}\hspace{-0.2em}/\Z)^*$
would represent the 
principal divisors on $\PP^1$, and the formula
$\div\big(f(y)\chi^u\big)=(\kQ,\kphi)^*\big(\div(f),u\big)$
of Theorem~\ref{thm-principDiv} would provide the exactness of the sequence.
However, for $p\in\PP^1\setminus\kP$, we have
$\pFan_p=\tail\pFan$, i.e.\ the corresponding $\Z$-summands of the first and
second place cancel each other.
\end{proof}

Let us eventually come to the description of global sections for $D\in \kFF\subseteq \tdiv$.
From \cite{tidiv} we recall the associated polyhedron
$$
\Box_D:= \conv \{u\in M_\Q\kst \langle\varrho,u\rangle \geq -\coef_\varrho D
\mbox{ for all } \varrho\in\kR\}
\vspace{0.5ex}
$$
where
$\coef_\varrho D$ denotes the coefficient of $D_\varrho$ inside $D$.
Moreover, there is a map
$\,D^*:\Box_D\to \kDiv_\Q\PP^1$ with
$\coef_p D^*(u):= 
\min\{\langle v,u\rangle +\coef_{(p,v)}D/\mu(v)\kst v\in\pFan_p(0)\}$.
As another direct consequence of Theorem~\ref{thm-principDiv}
one obtains 
$\kG\big(X,\cO_X(D)\big)(u)=\kG\big(\PP^1, \cO_{\PP^1}(D^*(u))\big)$,
cf.\ \cite{tidiv}.

%%%%%%%%%%%%%%%%%%%%%%%%%%%%%%%%%%%%%%%%%%%%%%%%%%%%%%%%%%%%%%%%%%%%%
%%%%%%%%%%%%%%%%%%%%%%%%%%%%%%%%%%%%%%%%%%%%%%%%%%%%%%%%%%%%%%%%%%%%%
%%%
%%%   Finite Coverings of $\PP^1$
%%%
%%%%%%%%%%%%%%%%%%%%%%%%%%%%%%%%%%%%%%%%%%%%%%%%%%%%%%%%%%%%%%%%%%%%%
%%%%%%%%%%%%%%%%%%%%%%%%%%%%%%%%%%%%%%%%%%%%%%%%%%%%%%%%%%%%%%%%%%%%%
\section{Finite Coverings of $\PP^1$}
\label{finiteCov}

%%%%%%%%%%%%%%%%%%%%%%%%%%%%%%%%%%%%%%%%%%%%%
%  Finite covers with abelian group actions
%%%%%%%%%%%%%%%%%%%%%%%%%%%%%%%%%%%%%%%%%%%%%

\subsection{Finite covers with abelian group actions}\label{fCgrp}
It turns out that the p-divisors of the Cox rings under investigation do not all live on $\PP^1$, but rather on special branched coverings of it. The latter come with an action of a finite abelian group $A$ whose quotient is $\PP^1$ again. We will present these coverings in the spirit of p-divisors, yet with different and unusual coefficients. For basic facts about group cohomology used in the subsequent text see \cite[Ch VII, p.109 ff]{locFields}.

Let us point out that algebraic actions of finite abelian groups as well as tori on normal varieties should be considered as special instances
of a yet to be developed general theory of actions of diagonalizable groups.

%%%%%%%%%%%%%%%%%%%%%%%%%%%%%%%%
%  $A$-divisors on $\PP^1$
%%%%%%%%%%%%%%%%%%%%%%%%%%%%%%%%

\subsection{$A$-divisors on $\PP^1$}
\label{A-divisors}
Let $\kA, \kAA$ be finite and mutually dual abelian groups, i.e.\
$\kAA:=\gExt^1_\Z(\kA,\Z)=\gHom_\Z(\kA,\Q/\Z)$.
Consider an element $\cE\in \kDiv^0_A \PP^1 := A \otimes \kDiv^0 \PP^1$ which is is equivalent to a group homomorphism $\cE:\kAA\to \kDiv_{\Q/\Z}^0\PP^1$. Here, we denote by $\kDiv^0\PP^1$ the group of divisors on $\PP^1$ which have degree zero. In the following we will construct an $(\#A)$-fold covering $q: C\to\PP^1$ out of these data:
\\[1ex]
For $a \in \kAA$ denote by $\wt{E}_a$ an arbitrary lifting
of the $\Q/\Z$-divisor $E_a:=\langle\cE,a\rangle$
to a $\Q$-divisor of degree zero.
Then, for $a,b\in\kAA$ there exist $f_{a,b}\in K(\PP^1)^*$ with
$$
\wt{E}_a + \wt{E}_b + \div(f_{a,b}) = \wt{E}_{a+b}.
$$
These functions $f_{a,b}$ are unique up to constants,
and because of $\gH^3(\kAA,\CC^*)=0$ ($\CC^*$ is a divisible and therefore injective abelian group), they can be chosen in such a way that
$$
f_{a,b} \cdot f_{a+b,c} = f_{a,b+c} \cdot f_{b,c}.
$$
Indeed, this follows from the fact that $s_{a,b,c}:= f_{b,c}\,f_{a+b,c}^{-1}\,f_{a,b+c}\,f_{a,b}^{-1}\in\CC^*$
is a 3-cocycle, and hence arises from elements 
$s_{a,b}\in\CC^*$. Thus, we are allowed to replace the functions $f_{a,b}$ by $f_{a,b}/s_{a,b}$, and the latter finally fulfill the above multiplicative rule. Introducing the map
$$
\begin{array}{ccc}
\cO_{\PP^1}(\wt{E}_a) \otimes \cO_{\PP^1}(\wt{E}_b)
&\longrightarrow&
\cO_{\PP^1}(\wt{E}_{a+b})
\\
f\hspace{0.6em}\otimes\hspace{0.6em} g & \mapsto&
f g\, f_{a,b}^{-1}
\end{array}
$$
induces an associative multiplication on
$\cA:=\oplus_{a\in\kAA} \cO_{\PP^1}(\wt{E}_a)$ which eventually gives rise to the definition of the curve $C:=\spec_{\PP^1}\cA$.

\begin{remark}
\label{rmk:indOfChoice}
1) The structure of the algebra $\cA$ does not depend on the choice
of the rational functions $f_{a,b}$. Indeed, let $g_{a,b}$ be another choice. Then, the quotient $s_{a,b}:=g_{a,b}/f_{a,b}\in\CC^*$ represents a 2-cocycle, and $\gH^2(\kAA,\CC^*)=0$ implies the existence of $s_a\in\CC^*$
with $s_{a,b}=s_a\,s_b\,s_{a+b}^{-1}$. Thus, the 
equation $g_{a,b}^{-1}\,s_a\,s_b= s_{a+b} \,f_{a,b}^{-1}$ shows
that the multiplication with $s_a$ yields an automorphism of $\kG(\PP^1, \wt{E}_a)$ identifying the ring structures induced from the elements $f_{a,b}$ and $g_{a,b}$, respectively.
\\[.5ex]
2) Similarly, $\cA$ does not depend on the choice of the rational
liftings $\wt{E}_a$ of the $\Q/\Z$-divisors $E_a$.
Indeed, replacing them by $\wt{E}_a'=\wt{E}_a+\div(s_a)$ for rational
functions $s_a\in K(\PP^1)^*$ leads to adjusted elements
$f'_{a,b}=f_{a,b}\,s_{a+b}\,s_a^{-1} s_b^{-1}$, and thus to 
mutually compatible isomorphisms $\kG(\PP^1, \wt{E}'_a)\stackrel{s_a}{\to}\kG(\PP^1, \wt{E}_a)$.
\end{remark}

\begin{theorem}
\label{th-smoothness}
The curve $C$ constructed above is smooth.
\end{theorem}

The proof of this theorem will be given in paragraph (\ref{smoothness}). In particular it will also allow for a detailed study of the monodromy of the covering.

On the other hand, suppose we are given a smooth projective curve $C$ together with an action of a finite abelian group $A$ such that its quotient is the projective line. Then with a little effort one can construct $A$-divisor on $\PP^1$ that recovers the whole picture.

%%%%%%%%%%%%%%%%%%%%%
%  Two special cases
%%%%%%%%%%%%%%%%%%%%%

\subsection{Cyclic coverings}
\label{cyclic}
The usual description of a cyclic $n$-fold covering of $\PP^1$ makes use of an invertible sheaf $\cL$ on $\PP^1$ together with an effective and reduced divisor $D\subseteq \PP^1$, and an isomorphism
$s: \cL^n \stackrel{\sim}{\to} \cO_{\PP^1}(-D)$, cf. e.g. \cite[Chapter 4.1]{posAG1}. Hence, the sheaf of $\cO_{\PP^1}$-algebras $\cA:=\oplus_{i=0}^{n-1} \cL^i$ obtains a multiplication
via $s:\cL^n \hookrightarrow \cO_{\PP^1}$.
If one chooses a divisor $E$ with $\cL = \cO_{\PP^1}(E)$,
then $s$ corresponds to a rational function such that
$nE +\div(s)=-D$.

This construction fits into the pattern of (\ref{A-divisors}) as
follows: Let $\kA = \ZZ/n\ZZ$. If $D$ is a divisor on $\PP^1$ with $n|\deg D$,
then this gives rise to a map $\Z/n\Z = \kAA \to \kDiv_{\Q/\Z}^0\PP^1$,
$1\mapsto \frac{1}{n}D$. Here, the ordinary divisor $E$
is used to obtain the liftings to the $\Q$-divisors
$\til{E}_{i}:= iE + \frac{i}{n}D$ for $\,i=0,\ldots,n-1$.
Finally, the functions 
$$
f_{i,j}:=\left\{\begin{array}{ll}
1 & \mbox{if } i+j\leq n-1, \mbox{ and}\\ 
s & \mbox{if } i+j\geq n
\end{array}\right.
$$
fulfill 
$\wt{E}_a + \wt{E}_b + \div(f_{a,b}) = \wt{E}_{a+b}$
as well as
$f_{a,b} \cdot f_{a+b,c} = f_{a,b+c} \cdot f_{b,c}$.

Note that the dual language is even simpler. Namely, the covering corresponds exactly to the divisor $D$ which is to be considered as an element of $\kDiv_{\ZZ/n\ZZ}^0 \PP^1$.

%%%%%%%%%%%%%%%%%%%%%%%%%%%%%%%%
%  Smoothness
%%%%%%%%%%%%%%%%%%%%%%%%%%%%%%%%
\subsection{Proof of Theorem \ref{th-smoothness}}
\label{smoothness}

%%%%%%%%%%%%%%%%%%%%%%%%%%
\subsubsection{}
\label{localPicture}
%%%%%%%%%%%%%%%%%%%%%%%%%%
Let $p \in \PP^1$ be a point with local parameter $t$, and identify $K(\PP^1)$ with $\CC(t)$. Furthermore, let $f_{a,b}=t^{k(a,b)}\keps_{a,b}$ with $\keps_{a,b}\in\CC[t]^*_{(t)}$ (i.e.\ $\keps_{a,b}(0)\neq 0$) where $k(a,b):=\ord_p f_{a,b} \in \ZZ$. Note that the units $\keps_{a,b}$ fulfill the same cocycle conditions as the functions $f_{a,b}$.
\\[1ex]
We may consider $E_a$ as a $\QQ$-divisor with coefficients in the interval $[0,1)$, i.e.\ it becomes the fractional part $\{\wt{E}_a\}$. We now define $\ell(a):=\coef_p E_a\in[0,1) \cap \QQ$ inducing a linear function $\ell:\kAA \to [0,1)\to\Q/\Z$, i.e.\ $\ell\in\kA$. Note that this is exactly the $p$-coefficient of the $\kA$-divisor $\cE$.
Moreover, we define
$$
\lambda(a,b):=\left\{\begin{array}{ll}
0 & \mbox{if } \ell(a) + \ell(b) < 1\\
1 & \mbox{if } \ell(a) + \ell(b) \geq 1,
\end{array}\right.
$$
which gives us $\ell(a) + \ell(b) = \ell(a+b)+\lambda(a,b)$. Finally we set $m(a):=\coef_p\rounddown{\wt{E}_a}$ which yields
$$
m(a) + m(b) + k(a,b) + \lambda(a,b)=m(a+b).
$$
Let us denote the local monomial generators of $\cO_{\PP^1}(\wt{E}_a)\subseteq\CC(t)$ by $s_a:=t^{-m(a)}$ ($a \in \kAA$). Then the set $\{s_a\kst a\in \kAA\}$ with $s_0=1$ becomes a basis of the stalk $\cA_p$ which is a free $\CC[t]_{(t)}$-module. Yet the multiplication of the $s_a$ in $\cA_p$ differs from the one of the monomials in $\CC(t)$:
$$
\begin{array}{rcl}
s_a \odot s_b 
&=& t^{-m(a)}\odot t^{-m(b)}\\
&=& t^{-m(a)-m(b)} f_{a,b}^{-1}\\
&=& t^{k(a,b)+\lambda(a,b)-m(a+b)}f_{a,b}^{-1}\\
&=& \keps^{-1}_{a,b}(t)\, t^{\lambda(a,b)}\, s_{a+b}.
\end{array}
$$
Thus, we obtain
$\cA_p=\CC[t,s_a \kst a\in\kAA]_{(t)}/\big(t^{\lambda(a,b)}\, s_{a+b}-\keps_{a,b}(t)\, s_as_b\big)$ with $s_0=1$.
\\[1ex]
Next we set $N:= \lcm\, \{\textnormal{denominators of } \ell(a) \kst a \in \kAA\}$, and $d(a):= N \cdot \ell(a) \in \ZZ$. It follows that $\gcd\, \{d(a) \kst a \in \kAA\} = 1$. Hence, there exist integers $c(a)$ such that $\sum_{a\in\kAA} c(a)\, d(a)=1$, and we define
$$\gamma:=\sum_{a\in\kAA} c(a)\cdot a\in\kAA.$$
Then we have
$$\ell(\gamma)=\ell\big(\sum_{a\in\kAA} c(a)\cdot a\big)
=\sum_{a\in\kAA} c(a)\,\ell(a) =1/N,$$
in other words $d(\gamma)=1$. Moreover, note that $d(a) = 0$ implies $\lambda(a,b) = 0$ for arbitrary $b \in \kAA$.
\\[1ex]
Now, consider a closed point $\underline{c}=(c_a\in\CC\kst a\in\kAA) \in C$ with $q(\underline{c})=p$. This is equivalent to the following two conditions on the coordinates of $\underline{c}$\,:
$$
c_a=0 \mbox{ for } d(a)\geq 1,
$$
and
$$
c_a\, c_b \,\keps_{a,b}(0)= c_{a+b} \neq 0
\mbox{ for } d(a)=0.
$$

It is not hard to check that both conditions imply the equations for the ideal given above. For the other direction, it suffices to prove the claim for $c_a=0$, $c_a \neq 0$ for $d(a)\geq 1$ and $d(a)=0$, respectively. Denoting the order of $a \in \kAA$ by $n_a$, we obtain the following equation for both cases:
$$
s_a^{n_a} = \big(\mbox{product of units in $\CC[t]_{(t)}$}\big)
\cdot t^\lambda \cdot s_{n_a a}
\mbox{ with }
\lambda=\left\{\begin{array}{ll}
\geq 1 & \mbox{if } d(a)\geq 1\\
0 & \mbox{if } d(a)=0
\end{array}\right.
$$
Thus, the claim follows from $s_{n_a a}=s_0=1$ which also yields $c_{n_a a}=1$.

%%%%%%%%%%%%%%%%%%%%%%%%%%%%
\subsubsection{}
\label{localIso}
%%%%%%%%%%%%%%%%%%%%%%%%%%%%
Having provided the necessary tools in the previous paragraph, we are now ready to define the following map:
$$
\xymatrix@R=0.0ex@C=2.2em{
\cA_p\;=\;\CC[t,s_a \kst a\in\kAA]_{(t)}/
\big(t^{\lambda(a,b)}\, s_{a+b}-\keps_{a,b}(t)\, s_as_b\big)
\ar[r]^-{\Phi_\eta} & 
\CC\{u\}\supseteq \CC[u]_{(u)}\\
{ }\\
t \ar@{|->}[r] & u^N\hspace{4em}\\
s_a \ar@{|->}[r] & \eta_a(u^N)\, u^{d(a)}\hspace{4em}
}
$$
with $\eta_a(u)\in \CC\{u\}^*$, such that
$\eta_a(0)=c_a$ for $d(a)=0$, and for all $a,b\in\kAA$
$$
\eta_a(u)\,\eta_b(u) \keps_{a,b}(u)=\eta_{a+b}(u).
$$

The existence of the functions $\eta_a(u)$ follows from the divisibility of the abelian group $\CC\{u\}^*$. The latter property is also the reason for why we have to enlarge the ring $\CC[u]_{(u)}$ to $\CC\{u\}^*$ (corresponding to the non-rationality of the curve $C$). Note that $\Phi_\eta(s_{\gamma})= \eta_{\gamma}(u^N)\, u$.
Observe that by choosing $\wt{E}_0 = 0$, we can arrange that $f_{a,0}=\keps_{a,0}=1$ and therefore $\eta_0(u)=1$.

$\Phi_\eta$ maps $t$ together with all differences $(s_a-c_a)$ into the ideal
$(u)\subseteq \CC\{u\}$, i.e.\ $\Phi_\eta$ induces a ring homomorphism
between the completions
$$
\hat{\Phi}_\eta:\widehat{\cO_{C,\underline{c}}}=
\widehat{(\cA_p)_{\underline{c}}}\longrightarrow\CC\{u\}.
$$
Since we have $\Phi_\eta(s_{\gamma}) = \eta_{\gamma}(u^N)\, u$, we conclude that $\hat{\Phi}_\eta$ is surjective.
\\[1ex]
We claim that $\hat{\Phi}_\eta$ is an isomorphism. To show that $\hat{\Phi}_\eta$ is injective we construct a map $\hat{\Psi}_\eta: \CC\{u\} \to \widehat{(\cA_p)_{\underline{c}}}$ such that $\hat{\Psi}_\eta \circ \hat{\Phi}_\eta \equiv \id$. Using the identity $\keps_{a,b} = \frac{\eta_{a+b}}{\eta_a \eta_b}$, we transform our initial relations $t^{\lambda(a,b)}\, s_{a+b}-\keps_{a,b}(t)\, s_as_b = 0$ in $\cA_p$ into 
$$
\eta_{a+b}(t)^{-1}s_{a+b}\,t^{\lambda(a,b)} - \eta_a(t)^{-1}s_a \cdot \eta_b(t)^{-1}s_b = 0\,.
$$
To simplify our notation we substitute the $s_a$ by the new variables $t_a:= \eta_a(t)^{-1}s_a$ which transforms the upper relation into 
$$
t_{a+b}\,t^{\lambda(a,b)} - t_a t_b = 0.
$$
Then we define the map $\hat{\Psi}_\eta: \CC\{u\} \to \widehat{(\cA_p)_{\underline{c}}},\, u \mapsto t_{\gamma}$, and by composition we obtain
$$
\big(\hat{\Psi}_\eta \circ \hat{\Phi}_\eta\big)(t) = t_{\gamma}^N, \textnormal{ and } \,\big(\hat{\Psi}_\eta \circ \hat{\Phi}_\eta\big)(t_a) = t_{\gamma}^{d(a)}\,.
$$
Thus, it remains to show that $t=t^N_{\gamma}$ and $t_a = t_{\gamma}^{d(a)}$ in the ring $\widehat{(\cA_p)_{\overline{c}}}$. This will be done in (\ref{finalProof}).

%%%%%%%%%%%%%%%%%%%%%%%%%%%%%%
\subsubsection{}
\label{intermediateProof}
%%%%%%%%%%%%%%%%%%%%%%%%%%%%%%
In an intermediate step, we prove that $d(a) = d(b)$ implies $t_a = t_b$ in $\widehat{(\cA_p)_{\overline{c}}}$. First let us restrict ourselves to those elements for which $d(\cdot) = 0$. Since $d(0) = 0$ and $t_0 = 1$ by definition, we have to show that $t_a = 1$ in $\widehat{(\cA_p)_{\overline{c}}}$. Yet, this does not involve variables $t_{a'}$ with $d(a') \geq 1$. Hence, it suffices to prove the identity inside the local ring $\big(\CC[t,t_a \kst d(a) = 0]/(t_{a+b} - t_a t_b)\big)_{\underline{c}}$\,, where $\underline{c}$ is the point that corresponds to the maximal ideal $(t,t_a-1)$. Observe that 
$$
\CC[t,t_a \kst d(a) = 0]/(t_{a+b}-t_a t_b) = \CC[t][\ker l \subseteq \kAA]
$$
is a commutative group algebra. Furthermore, denoting the order of $\ker l$ by $n$, we have the following equalities in $\CC[t][\ker l]_{\underline{c}}$\,:
$$
(t_a -1) \cdot (\textnormal{unit}) = \prod_{\xi_k \in \mu_n} (t_a - \xi_k) = t_a^n -1 = t_0 - 1 = 0\,. 
$$

Consider now elements $a \in A$ with $d(a) \geq 1$. Assume that $d(a) = d(b)$, and set $c: = a-b$. From $d(c)=0$ we deduce that $\lambda(b,c) = 0$ and $t_c = 1$. Thus, the relation $t_a-t_bt_c=0$ from the ideal proves our claim.

%%%%%%%%%%%%%%%%%%%%%%%%%%
\subsubsection{}
\label{finalProof}
%%%%%%%%%%%%%%%%%%%%%%%%%%
Let us now come back to the two equations given at the end of (\ref{localIso}). By the relations of the $t_a$, we have that $t_{k\gamma} = t_\gamma^k$ for $0\leq k <N$ and $t\cdot t_{N\gamma} = t_{\gamma}^N$. The first equality in (\ref{localIso}) now follows from $t_{N\gamma}=1$ ($d(N\gamma) = 0$). The second one holds due to $d(a) = d\big(d(a)\cdot\gamma\big)$ and the fact that $t_a = t_{d(a)\gamma} = t_{\gamma}^{d(a)}$ which follows from (\ref{intermediateProof}). This completes the proof of Theorem~\ref{th-smoothness}.

%%%%%%%%%%%%%%%%%%%%%%%%%%%%%%%%%%%%%%%%%%%%%%%%%%%%%%%%%%%%%%%%%%%%%
%%%%%%%%%%%%%%%%%%%%%%%%%%%%%%%%%%%%%%%%%%%%%%%%%%%%%%%%%%%%%%%%%%%%%
%%%
%%%   The P-divisor of the Cox ring
%%%
%%%%%%%%%%%%%%%%%%%%%%%%%%%%%%%%%%%%%%%%%%%%%%%%%%%%%%%%%%%%%%%%%%%%%
%%%%%%%%%%%%%%%%%%%%%%%%%%%%%%%%%%%%%%%%%%%%%%%%%%%%%%%%%%%%%%%%%%%%%
\section{The P-Divisor of the Cox Ring}\label{pDivCox}

%%%%%%%%%
%  Qphi
%%%%%%%%%

\subsection{The map $(\kQ,\kphi)$}
\label{Qphi}
Let $X=X(\pFan)$ be given by a complete divisorial fan
$\pFan=\sum_{p\in\PP^1} \pFan_p \otimes [p]$ on $\PP^1$.
Recall that the map $\kQ$ from (\ref{globX}) becomes surjective after
tensoring with $\Q$. Over $\Z$, however, we define the following two lattices $\tN:=\ker \kQ\subseteq \kF$, $\kL:=\im\kQ\subseteq\kZP$, and we denote by $A$ the finite abelian group
$A:=\coker\kQ$. Note that all these groups can be assembled in the following exact sequences
$$
\;0 \to \tN \to \kF \stackrel{\kQ}{\to} \kZP \to \kA \to 0
\hspace{1.3em}\mbox{and}\hspace{1.3em}
\;0 \to \kL \to \kZP \to \kA\to 0.
$$
Furthermore, we choose a section $\ks:\kL \hookrightarrow \kF$, and by abuse of notation, we denote its rational extension $\ks:\kZP\to\kFQ$ by the very same letter.

Although the next lemma is already a consequence of Corollary~\ref{cor-ClXP1}, we believe that the following direct and combinatorial proof sheds some additional light on the whole situation.

%%%%%%%%%%%%%%%
% Lemma Qsurj
%%%%%%%%%%%%%%%
\begin{lemma}
\label{lem-Qsurj}
The map $(\kQ,\kphi)_\Q:\kFQ\to(\kQP)\oplus N_\Q$ is surjective.
\end{lemma}

\begin{proof}
Since $\deg(\pDiv)=\sum_p\conv \pDiv_p(0) + \tail(\pDiv)\subsetneq \tail(\pDiv)$ for a single $\pDiv \in \pFan$, every ray $\varrho \in \tail(\pDiv)(1)$ either belongs to $\kR$ (meaning that $\varrho \cap \deg\pDiv = \emptyset$), or $\varrho$ intersects $\sum_p\conv \pDiv_p(0)$. Thus, every ray $\varrho \in \tail(\pFan)(1)$ either belongs to $\kR$ or it intersects $\sum_p\conv \pFan_p(0)$ away from the origin.
This fact means that non-zero elements of each ray of the tailfan
occur in the image of the map $\kphi_\Q: Q^{-1}(0) \to N_\Q$, i.e.\
it is surjective.

On the other hand, every slice contains at least one vertex which implies that $\kQ_\Q:\Q^{\kV}\to\kQP$ is also surjective.
\end{proof}

%%%%%%%%%%%%%%%%%%%%%%%%%%%%%%%
%  The main polyhedral objects
%%%%%%%%%%%%%%%%%%%%%%%%%%%%%%%

\subsection{The main polyhedral objects}
\label{somePol}
We briefly recall the polyhedral objects we defined in (\ref{c1KL}), namely the polytopes
$$
\Delta_p^c = \conv\{e(v)/\mu(v)\kst v\in \pFan_p(0)\}
\subseteq Q^{-1}(\ko{e(p)})
\subseteq\kFQ,
$$
and the polyhedral cone
$$
\sigma \;=\; \Q_{\geq 0}\cdot\prod_{p\in\kP} \Delta_p^c 
\;+\; \kpR 
\;=\;Q^{-1}(0)\cap \kpVR
\;\subseteq\;\tN_\Q.
$$
From these data we build
$$
\begin{array}{rcl}
\Delta_p & :=& \Delta_p^c + \sigma \;=\; 
Q^{-1}(\ko{e(p)})\cap \kpVR
\hspace{2em}\mbox{and}
\\[1.0ex]
\wt{\Delta}_p &:=& \Delta_p - \ks(\ko{e(p)})\;\subseteq\; \tN_\Q.
\end{array}
$$
Moreover, following Section~\ref{finiteCov}, the cokernel $\pi:\kZP\to \kA$ of $\kQ$ from (\ref{Qphi}) yields an $A$-divisor of degree zero, i.e.\ a finite covering $q:C\to\PP^1$. Note that for each $c\in C$
the ramification index $e_c \in \Z_{\geq 1}$ only depends on the image $p=q(c) \in \PP^1$.
While $p \in \PP^1 \setminus \kP$ leads to $e_c=1$, the ramification indices of points in the fibers over elements $p \in \kP$ equal the order of $\pi(\ko{e(p)}) \in \kA$.

%%%%%%%%%%%%%%%%%%%%%
% theorem th-mainTh
%%%%%%%%%%%%%%%%%%%%%
\begin{theorem}
\label{th-mainTh}
$\Cox(X)$ corresponds to the p-divisor
$\cD_{\Cox} = \sum_{c\in C} e_c\cdot \wt{\Delta}_{q(c)}\otimes [c]$ on $C$.
\end{theorem}

The rest of this section is devoted to the proof of this
theorem. Clearly, in the case that $\cl(X)$ is torsion free, the group $\kA$ is trivial. Hence, the map $q$ is an isomorphism, and Theorem~\ref{th-mainTh} implies Theorem~\ref{th-mainThIntro} from the introduction. The shift mentioned there is the same shift as in the above definition of $\wt{\Delta}_p$.

\begin{remark}
\label{rmk:properness}
The polyhedral divisor defined in Theorem~\ref{th-mainTh} has a complete locus. It is proper since $\deg \cD_{\Cox}$ is properly contained in the tailcone, as can easily be seen from the definition (cf. also \cite[Example 2.12]{tvar_1}).

Moreover, the handy description of $\cD_{\Cox}$ makes it possible to sometimes quickly check some of the singularity criteria for $X(\cD_{\Cox})$ presented in \cite[Section 5]{normalSing}.
\end{remark}

%%%%%%%%%%%%%%%%%%%
%  torsionDiagram
%%%%%%%%%%%%%%%%%%%

\subsection{The big diagram}
\label{torsionDiagram}
Dualizing the exact sequences from (\ref{Qphi}), we obtain maps like
$\kss:\kFF\to \kLL\subseteq \kQQP$,
a torsion free lattice $\tM:=\gHom(\tN,\Z)$, and the dual
finite abelian group $\kAA:=\gHom_\Z(A,\Q/\Z)=\gExt^1_\Z(A,\Z)$
as outlined in (\ref{A-divisors}).
The whole picture may be visualized in the following commutative diagram:
$$
\xymatrix@R=2.5ex@C=1.2em{
&0 &&& 0 &&& 0\\
%%%%%%%%%%%%%%%%%%%%%%%%%%%%%%%%%%
&
\kAA\, \ar@{^(->}[rrr]\ar[u]&&&
\cl(X) \ar@{=}[rrr]\ar[u] &&&
\cl(X)\ar[u]\\
%%%%%%%%%%%%%%%%%%%%%%%%%%%%%%%%%%
&&&&&0\\
%%%%%%%%%%%%%%%%%%%%%%%%%%%%%%%%%%
0\ar[r]&
\kLL\ar[uu]\ar[rrr]&&&
\kFF\ar[uu]\ar[rr]^-{\kRR}
\ar@/_1.1pc/[lll]_-{\kss}&&
\tM\ar[rr]\ar[lu]
\ar@/^1.1pc/[ll]_-{\kt}&&
0\\
%%%%%%%%%%%%%%%%%%%%%%%%%%%%%%%%%%
0\ar[rr]&&
\kZZP\ar[rr]^-{\kQQ}\ar@{_(->}[lu]&&
\kFF\ar@{=}[u]\ar[rrr]^-{\ktRR}&&&
\ttM\ar[uuu]\ar[r]\ar[lu]_-{\pi}\ar@{=>}[ru]&
0\\
%%%%%%%%%%%%%%%%%%%%%%%%%%%%%%%%%%
&&&&&&&&\kAA\ar[lu]\\
%%%%%%%%%%%%%%%%%%%%%%%%%%%%%%%%%%
0\ar[r]&
\kZZP\ar[uuu]\ar[rrr]\ar@{=}[ruu]&&&
\kZZP\oplus M\ar[rrr]\ar[uu]^-{\kQQ}_-{\kphii}
\ar@/^0.8pc/[uurrr]^-{\Phi_1}\ar@/_0.8pc/[uurrr]_-{\Phi_2} &&&
M\ar[r]\ar[uu]^-\beta&
0 & 0\ar[lu]\\
%%%%%%%%%%%%%%%%%%%%%%%%%%%%%%%%%%
&0\ar[u]&&0\ar@{=>}[ru]&0\ar[u]&&&0\ar[u]
}
$$
The row containing $\kLL$ and $\tM$ is an exact sequence of free abelian groups. Since $\kLL$ is a saturated sublattice, we have that
$\kLL \cap \big(\kZZP \oplus M\big) = \kZZP$
within $\kFF$. In particular, $\kAA \to \cl(X)$ is injective.
Observe that if $\ttM$ denotes the cokernel of $\kQQ$, then
we even know that 
$$
\kAA=\Tors(\ttM)\subseteq\Tors\big(\cl(X)\big).
$$
Hence, if $\cl(X)$ is torsion free the finite groups $\kAA$ and $\kA$ become trivial. While the maps $\kRR$, $\ktRR$, and $\pi$ are explained by the diagram and the remarks made before, the map $\kt$ is supposed to be the section induced from $\kss$, i.e.\ $\kQQ \kss + \kt \kRR = \id$. We will add some details to the remaining maps $\Phi_1,\Phi_2$ appearing in the lower right corner in (\ref{ttMCox}).

%%%%%%%%%%%%%%%%%%%%%%%%%%%%%%%%%%%%
%  The $\ttM$-grading of $\Cox(X)$
%%%%%%%%%%%%%%%%%%%%%%%%%%%%%%%%%%%%

\subsection{The $\ttM$- and $\tM$-grading of $\Cox(X)$}
\label{ttMCox}
Let $f:\kZZP\to K(\PP^1)^*$, $E\mapsto f_E$ be a linear map
such that $\div(f_E)=E$ for $E\in\kZZP \subseteq \kDiv^0\PP^1$. We can then deduce from Theorem~\ref{thm-principDiv} that the map
$(f,\chi):(E,u)\mapsto f_E(y)\chi^u\in K(X)$ is a lifting of
$$(\kQQ,\kphii):\kZZP\oplus M \to \kFF\subseteq\tdiv X.$$
Now, Corollary~\ref{cor-ClXP1} states that the latter map gives a presentation of $\cl(X)$. So we are exactly in the setting of \cite[Section 2]{tvarcox} used for the definition of the Cox ring under the presence of torsion in the class group. In particular,
$$
\Cox(X) = \oplus_{D\in\kFF} \kG(X,\,\cO_X(D))\;
\big/
\big(1-f_E\chi^u\kst (E,u)\in \kZZP\oplus M\big).
$$
The $T$-action on $X$ induces an $M$-grading of all vector spaces $\kG(X,\,\cO_X(D))$. On the whole, the sum $\oplus_{D\in\kFF} \kG(X,\,\cO_X(D))$ admits a $(\kFF\oplus M)$-grading.

In addition, to make the ideal 
$\big(1-f_E\chi^u\kst (E,u)\in \kZZP\oplus M\big)$ homogeneous,
we have to set the degrees coming from
$\kZZP \oplus M \hookrightarrow \kFF \oplus M$ equal to zero. Since this embedding is represented by the matrix
$$
\left(\begin{array}{@{}c@{\;}c@{}}
\kQQ & 0 \\ \kphii & \id_M
\end{array}\right),
$$
and since its two columns correspond to the two paths
$\Phi_1,\Phi_2$ in the big diagram of (\ref{torsionDiagram}),
the corresponding exact sequence
$$
0 \to \kZZP\oplus M \to \kFF\oplus M \to \ttM \to 0
$$
shows that $\Cox(X)$ carries a natural $\ttM$-grading.
To avoid torsion in the grading group, we downgrade
$\Cox(X)$ via $\pi:\ttM\surj\tM$.

%%%%%%%%%%%%%%%%%%%%%%%%%%%%%%%%%%%%%%%%%%%
%  Relating polyhedra to the Cox sheaf
%%%%%%%%%%%%%%%%%%%%%%%%%%%%%%%%%%%%%%%%%%%

\subsection{Relating the polyhedra $\Delta_p$ and $\wt{\Delta}_p$ to the Cox sheaf}
\label{usePol}
The multiplication with $\chi^{-u}$ induces isomorphisms
between the graded pieces of the above spaces of sections
$$
\chi^{-u}:\kG(X,\cO_X(D))(u)\stackrel{\sim}{\to}
\kG(X,\cO_X(D+\kphii(u)))(0).
$$
In particular, to describe $\Cox(X)$ we may forget all
non-zero degrees $u\in M$:
$$
\begin{array}{rcl}
\Cox(X) &=& 
\oplus_{D\in\kFF,\,u\in M} \kG\big(X,\cO_X(D)\big)(u)\,
\big/ \big(1-f_E\chi^u \kst E\in \kZZP\big)
\\[0.5ex]
&=& \oplus_{D\in\kFF} \kG\big(X,\cO_X(D)\big)(0\in M)\,
\big/ \big(1-f_E \kst E\in \kZZP\big).
\end{array}
$$
Using the notation and the remarks at the end of (\ref{globX})
in the first place, and then our knowledge about the vertices of 
$\Delta_p^c$ (and hence $\Delta_p$) from (\ref{somePol}), we can continue with
$$
\begin{array}{rcl}
\Cox(X) &=& 
\oplus_{D\in\kFF} \kG\big(\PP^1, D^*(0)\big)\,
\big/ \big(1-f_E \kst E\in \kZZP\big)
\\[0.5ex]
&=& \oplus_{D\in\kFF} 
\kG\big(\PP^1, \sum_p \min\langle  \Delta_p, D\rangle [p]\big)\,
\big/ \big(1-f_E \kst E\in \kZZP\big).
\end{array}
$$
Note that multiplication with the elements $f_E$ yields isomorphisms 
$$\kG(X,\cO_X(D))(0)\stackrel{\sim}{\to}
\kG(X,\cO_X(D-\kQQ(E)))(0).$$ 
Thus, exactly one summand per fine degree from $\ttM$ is needed in the previous direct sum. To make such a choice, we use an arbitrary set theoretical section $\psi: \kAA \hookrightarrow \kLL \subseteq \kFF$ of $\kLL\surj\kAA$ in the following way:

Let $w\in\tM$. Then $\kt(w) \in \kFF$ is a divisor which represents the single preimage $\ktRR\kt(w)\in\ttM$ of $w$ via $\pi$.
We deduce that $\pi^{-1}(w)=\ktRR\kt(w)+ \kAA\subseteq \ttM$,
and these elements are represented by the divisors
$D:=\kt(w) + \psi(a)\subseteq\kFF$ for $a\in\kAA$.
\\[1ex]
Now we use the orthogonality relations
$\langle \tN,\kLL\rangle=\langle\kL,\tM\rangle=0$
to obtain the following equalities:
$$
\begin{array}{rcl}
\min\langle  \Delta_p, \kt(w)+\psi(a)\rangle
&=&
\min\langle  \wt{\Delta}_p, \kt(w)+\psi(a)\rangle + 
\langle \ks(\ko{e(p)}),\, \kt(w)+\psi(a)\rangle
\\[0.5ex]
&=&
\min\langle  \wt{\Delta}_p, w\rangle + 
\langle \ko{e(p)},\, \kss\kt(w) +\kss(\psi(a))\rangle
\\[0.5ex]
&=&
\min\langle  \wt{\Delta}_p, w\rangle + 
\langle \ko{e(p)},\, \psi(a)\rangle.
\end{array}
$$
Since we have $\ko{e(p)}\in\Z^p/\Z\subseteq\Q^p/\Q=\kL_\Q$,
we know that $\langle \ko{e(p)},\, \psi(a)\rangle\in\Q$.
Thus, defining 
$$
\pDiv:=\sum_p \wt{\Delta}_p\otimes [p]
$$
and identifying $\psi(a)\in\kLL\subseteq\kLL_\Q=\kQQP$
as the $\Q$-divisor 
$\sum_p \langle \ko{e(p)},\, \psi(a)\rangle\, [p]$
on $\PP^1$,
% $\wt{E}_a=\psi(a)$
we obtain $\Cox(X)$ as the global sections of the sheaf
$$
\;\CoxSheaf(X):= \oplus_{w\in\tM,\,
a\in\kAA}\cO_{\PP^1}\big(\pDiv(w)+\psi(a)\big).
$$
Observe that the ring structure of this sheaf not only makes use of the relations $\pDiv(w)+\pDiv(w')\leq\pDiv(w+w')$ but also of the rational functions $f_{\psi(a)+\psi(b)-\psi(a+b)}$ obtained from $\psi(a)+\psi(b)-\psi(a+b)\in \kZZP \subseteq \kLL$ for $a,b\in\kAA$.

%%%%%%%%%%%%%%%%%
%  The covering
%%%%%%%%%%%%%%%%%

\subsection{The covering}
\label{useCov}
The covering $q:C\to\PP^1$ is induced from the map $\pi:\kZP\to\kA$ which is the cokernel of $\kQ$. This means that the elements $a\in\kAA=\gHom(A,\Q/\Z)$  provide  $\Q/\Z$-divisors $E_a = \sum_{p\in\PP^1}(a\pi)(\ko{e(p)})\,[p]$.
According to (\ref{A-divisors}) the sheaf $\cA=q_*\cO_C$ is then
built from their liftings to some $\wt{E}_a\in\kDiv_{\Q}^0\PP^1$.
\\[1ex]
Moreover, the fact that $\psi(a)\in\kLL$ maps to $a\in\kAA$ leads to the following commutative diagram with exact rows:
$$
\xymatrix@R=3.0ex@C=2.2em{
0 \ar[r] & \kL \ar[r]\ar[d]^-{\psi(a)} & 
\kZP \ar[r]^-{\pi} \ar[d]^-{\psi(a)}& 
\kA \ar[r]\ar[d]^-{a} & 0
\\
0 \ar[r] & \Z \ar[r] & \Q \ar[r] & \Q/\Z \ar[r] & 0
}
$$
It follows that $\langle \ko{e(p)},\psi(a)\rangle = (a\pi)(\ko{e(p)})$ in $\Q/\Z$, i.e.\ we can choose $\wt{E}_a=\psi(a)$. Furthermore, the elements $f_{a,b}$ which
provide the multiplicative structure of $\cA$ can be defined as
$f_{a,b} = f_{\psi(a)+\psi(b)-\psi(a+b)}.$
Thus, using the relation 
$ \cO_{\PP^1}(D+E) \supseteq \cO_{\PP^1}(D)\otimes\cO_{\PP^1}(E)$ 
for the $\Q$-divisors $D$ and $E$, we obtain 
$$
\CoxSheaf(X)
\supseteq
\Big(\oplus_{w\in\tM}\cO_{\PP^1}(\pDiv(w))\Big) \otimes
\Big(\oplus_{a\in\kAA}\cO_{\PP^1}(\wt{E}_a)\Big)
= \cO_{\PP^1}(\pDiv)\otimes q_*\cO_C.
$$
Both sides are $\tM$-graded, and for saturated $w$
(meaning that $\min\langle  \wt{\Delta}_p, w\rangle\in\Z$),
the left hand summand $D=\pDiv(w)$ becomes an ordinary $\Z$-divisor. In this case, the inclusion actually becomes an equality.
\\[1ex]
For any $\Q$-divisor $D(=\pDiv(w))$ we know that
$\rounddown{q^*D} \geq  q^*\rounddown{D}$, hence
$$
q^*\cO_{\PP^1}(D)=q^*\cO_{\PP^1}(\rounddown{D})
=\cO_{\PP^1}(q^*\rounddown{D})\subseteq
\cO_{\PP^1}(\rounddown{q^*D})=\cO_{\PP^1}(q^*D),
$$
and therefore
$$
\cO_{\PP^1}(D)\otimes q_*\cO_C = 
q_*q^*\cO_{\PP^1}(D)\subseteq q_*\cO_{\PP^1}(q^*D).
$$
Combining these relations with the previous inclusion regarding the Cox ring, we finally arrive at
$$
\CoxSheaf(X)
\supseteq
\cO_{\PP^1}(\pDiv)\otimes q_*\cO_C
\subseteq
q_*\cO_{\PP^1}(q^*\pDiv)
$$
with equality for both sides for saturated degrees $w\in\tM$.
Since the two outer sheaves are normal we have that
$\CoxSheaf(X)=q_*\cO_{\PP^1}(q^*\pDiv)$, which then gives
$$
\Cox(X)=\kG(\PP^1,\CoxSheaf(X))=
\kG(\PP^1,q_*\cO_{\PP^1}(q^*\pDiv))
= \kG(C,\cO_C(q^*\pDiv)).
$$
This establishes $\cD_{\Cox} = q^*\cD$ and completes the proof.
\qed

%%%%%%%%%%%%%%%%%%%%%%%%%%%%%%%%%%%%%%%%%%%%%%%%%%%%%%%%%%%%%%%%%%%%%
%%%%%%%%%%%%%%%%%%%%%%%%%%%%%%%%%%%%%%%%%%%%%%%%%%%%%%%%%%%%%%%%%%%%%
%%%
%%%   Examples
%%%
%%%%%%%%%%%%%%%%%%%%%%%%%%%%%%%%%%%%%%%%%%%%%%%%%%%%%%%%%%%%%%%%%%%%%
%%%%%%%%%%%%%%%%%%%%%%%%%%%%%%%%%%%%%%%%%%%%%%%%%%%%%%%%%%%%%%%%%%%%%
\section{Examples}
\label{ex}

%%%%%%%%%%%%%%%%%%%%%%%%%%%%%%%%%%%%%%%%%%
\subsection{Cox rings and degenerations} 
\label{subsec:coxDegenerations}
%%%%%%%%%%%%%%%%%%%%%%%%%%%%%%%%%%%%%%%%%%

Using the degeneration techniques developed in \cite{phdNathan,defrat_tvar}, we construct a toric degeneration of $X(\fan) = \PP(\Omega_{\PP^2})$ from Example~\ref{ex-Omega} to the projective cone over the del Pezzo surface of degree six (see also \cite[Example 5.1]{candiv}) denoted by $X' = X(\fan')$. We have $\tail \fan' = \Sigma$, $Y' = \PP^1$ with the relevant slices pictured in Figure~\ref{fig:degenerationCotang}. Observe that the marking is the same as for $\pFan$, namely $\Sigma(1) \cup \Sigma(2)$.

\begin{figure}[htbp]
  \centering
  \subfigure[$\fan'_0$]{\degenerationZero}
  \subfigure[$\fan'_\infty$]{\tailFan}
  \caption{Fansy divisor of a toric degeneration of $\PP(\Omega_{\PP^2})$.}
  \label{fig:degenerationCotang}
\end{figure}

Since $X(\fan')$ is toric we know that its Cox ring is a polynomial ring. Applying our recipe, we can see from Figure~\ref{fig:degenerationCotang} that the compact part $\wt{\Delta}^c_{0}$ is a five-dimensional simplex. 

Yet, performing the analogous degeneration on the level of Cox rings, i.e.\ adding up all polyhedral coefficients of the polyhedral divisor described in Example~\ref{ex:coxDivisorCotang}, gives a (toric) $\CC$-algebra which is \emph{not} a polynomial ring. Observe that the compact part of the only non-trivial polyhedral coefficient is the Minkowski sum of three edges, i.e.\ a three-dimensional cube.

%%%%%%%%%%%%%%%%%%%%%%%%%%%%%%%%%%%%%%%%%%%%%%%%%%%%%%%%%%
\subsection{Cox rings of log del Pezzo $\Cstar$-surfaces}
\label{subsec:coxSurfaces}
%%%%%%%%%%%%%%%%%%%%%%%%%%%%%%%%%%%%%%%%%%%%%%%%%%%%%%%%%%
We conclude the paper with two examples taken from the classification list of Gorenstein log del Pezzo $\Cstar$-surfaces in \cite{candiv}.

\begin{example}
Let $X(\fan)$ be the Gorenstein del Pezzo $\Cstar$-surface of degree three with singularity type $E_6$. It has two elliptic fixed points, i.e.\ $\cR = \emptyset$. The marked fansy divisor $\fan$ is illustrated in Figure~\ref{fig:delPezzoE6}. 

\begin{figure}[htbp]
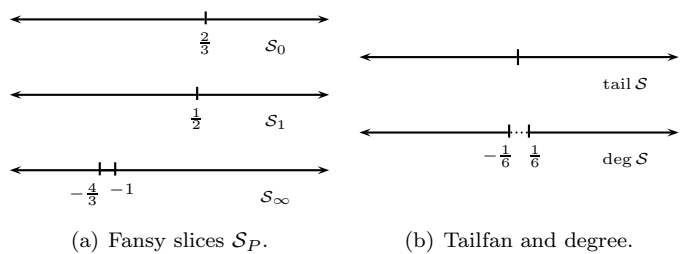

  \centering
  \subfigure[Fansy slices $\fan_P$.]{\dPfan}
  \subfigure[Tailfan and degree.]{\dPfanPlus}
  \caption{Fansy divisor of a Gorenstein log del Pezzo $\Cstar$-surface of singularity type $E_6$.}
  \label{fig:delPezzoE6}
\end{figure}

The divisor class group $\cl\big(X(\fan)\big)$ is torsion free of rank one, so $\kA = \kAA = 0$ and the covering map $q:\PP^1 \to \PP^1$ is the identity. Using the matrix
\[Q = \left( \begin{array}{rrrr} 3 & 0 & -3 & -1 \\ 0 & 2 & -3 & -1 \end{array}\right)\] 
and choosing a suitable section $s$, we obtain a tailcone $\sigma$ which is generated by the rays $(0,1)$ and $(-2,1)$. Furthermore,
\[\begin{array}{lll} \wt{\Delta}_0 & = & (0,-1/3) + \sigma, \\[1mm] \wt{\Delta}_1 & = & (0,1/2) + \sigma, \\[1mm] \wt{\Delta}_\infty & = & \ovl{\big((0,0),(-1/3,0)\big)} + \sigma, \end{array}\]
which eventually gives
\[\cD_{\Cox} = \wt{\Delta}_0 \otimes [0] + \wt{\Delta}_1 \otimes [1] + \wt{\Delta}_\infty \otimes [\infty]\,.\]
\end{example}

\begin{example}
Finally, consider the Gorenstein log del Pezzo $\Cstar$-surface of degree one with singularity type $E_6$ and $A_2$. It has two elliptic fixed points, so $\cR = \emptyset$ as before. The marked fansy divisor $\fan$ is pictured in Figure~\ref{fig:delPezzoE6A2}. 

\begin{figure}[htbp]
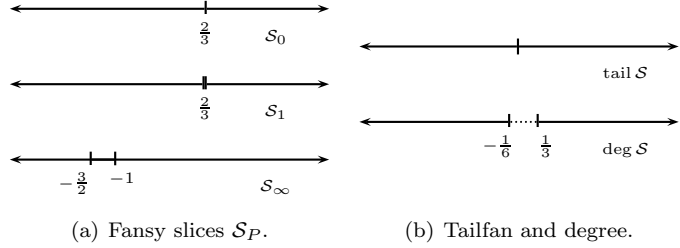

  \centering
  \subfigure[Fansy slices $\fan_P$.]{\dPtorsionFan}
  \subfigure[Tailfan and degree.]{\dPtorsionFanPlus}
  \caption{Fansy divisor of a Gorenstein log del Pezzo $\Cstar$-surface with singularity type $E_6$ and $A_2$.}
  \label{fig:delPezzoE6A2}
\end{figure}

The divisor class group $\cl\big(X(\fan)\big)$ is isomorphic to $\ZZ \oplus \ZZ/3\ZZ$. As above, using the matrix
\[Q = \left( \begin{array}{rrrr} 3 & 0 & -2 & -1 \\ 0 & 3 & -2 & -1 \end{array}\right)\] 
and a suitable section $s$, we derive our combinatorial data of the polyhedral divisor. Namely, the tailcone $\sigma$ is generated by the rays $(3,2)$ and $(0,1)$, whereas the polyhedral coefficients are given by
\[\begin{array}{lll} \wt{\Delta}_0 & = & \sigma, \\[1mm] \wt{\Delta}_1 & = & (0,1/3) + \sigma, \\[1mm]\wt{\Delta}_\infty & = & \ovl{\big((0,0),(1/2,0)\big)} + \sigma. \end{array}\]

Moreover, we also see that $A \cong \ZZ/3\ZZ$ gives rise to a 3:1-covering $q: \PP^1 \to \PP^1$ which is branched over $0$ and $1$. The polyhedral divisor describing the total coordinate space $\spec \Cox \big(X(\fan)\big)$ is thus given by
\[\cD_{\Cox} = 3 \cdot \wt{\Delta}_0 \otimes [c_0] + 3 \cdot \wt{\Delta}_1 \otimes [c_1] + \sum_{i=1}^3 \wt{\Delta}_\infty \otimes [c^i_\infty],\]
where $c^i_p \in q^{-1}(p)$ denotes a preimage of $p \in \PP^1$.
\end{example}

\bibliographystyle{alpha}
\bibliography{tcoxFinal}

\end{document}